\documentclass[11pt, reqno]{amsart}
\usepackage[margin=1in]{geometry}
\usepackage{amssymb,amsmath,amsthm,thmtools}
\usepackage{bbm}
\usepackage{enumitem}
\usepackage{tikz}
\usepackage{ytableau}
\usepackage[colorlinks=true,citecolor=blue,linkcolor=blue,backref=page]{hyperref}
\usepackage{cleveref}

\numberwithin{equation}{section}
\newtheorem*{introthm*}{Theorem}
\newtheorem{theorem}{Theorem}[section]
\newtheorem{lemma}[theorem]{Lemma}
\newtheorem{corollary}[theorem]{Corollary}
\theoremstyle{definition}
\newtheorem{definition}[theorem]{Definition}
\newtheorem{def-prop}[theorem]{Definition-Proposition}
\newtheorem{remark}[theorem]{Remark}
\newtheorem{example}[theorem]{Example}
\newtheorem*{acknowledgement}{Acknowledgements}
\newtheorem{notation}[theorem]{Notation}

\DeclareMathOperator{\Ass}{Ass}
\DeclareMathOperator{\Min}{Min}
\DeclareMathOperator{\maxass}{maxAss}
\DeclareMathOperator{\GL}{GL}
\DeclareMathOperator{\Gr}{Gr}

\newcommand{\ZZ}{{\mathbb Z}}
\newcommand{\NN}{{\mathbb N}}
\newcommand{\RR}{{\mathbb R}}
\newcommand{\kk}{{\mathbbm k}}
\def\P{{\mathcal P}}
\def\a{{\bf a}}
\def\b{{\bf b}}
\newcommand{\ba}{\mathbf{a}}
\renewcommand{\l}{\lambda}
\newcommand{\s}{\sigma}

\begin{document}

	\title{Symbolic Powers and Asymptotic Invariants of GL-Invariant Ideals}

	\author{Sankhaneel Bisui}
	\address{Department of Mathematics\\ Indian Institute of Science Education and Research Bhopal \\ Bhopal Bypass Road \\ Bhauri, Bhopal 462066 \\ Madhya Pradesh, India}
	\email{sankhaneel@iiserb.ac.in}

	\author{Alexandra Seceleanu}
	\address{ Department of Mathematics\\ University of Nebraska--Lincoln\\
		203 Avery Hall\\
		Lincoln, NE 68588, USA}
	\email{aseceleanu@unl.edu}


	\begin{abstract}
	This work concerns ideals invariant under the action of the group of linear base changes on a generic matrix; we call these GL-invariant ideals. We determine the ordinary powers, their saturations with respect to determinantal ideals, and the symbolic powers of GL-invariant ideals. We give explicit formulas for asymptotic invariants known as (skew) Waldschmidt constants and asymptotic resurgence, which measure the growth of these families and  compare  the ordinary and symbolic topologies. We also prove that the generalized symbolic Rees algebras associated with these ideals are Noetherian.\end{abstract}

	\maketitle

	\setcounter{tocdepth}{1}
	\tableofcontents

\section{Introduction}\label{s: intro}

Let $\kk$ be a field of characteristic zero, let
$X=(x_{ij})$ be a generic $r\times n$ matrix with $r\leq n$, and set
$R=\kk[X]$.  The group
\[
G=\GL_r(\kk)\times\GL_n(\kk)
\]
acts on $R$ by changes of basis, with
\[
(A,B)\cdot X=A^{-1}XB.
\]
We study the ideals of $R$ that are invariant under this action.

The representation-theoretic structure of $R$ gives a combinatorial
description of its $G$-invariant ideals.  Its irreducible isotypic components
are indexed by Young diagrams $\sigma$ having at most $r$ columns.  We denote
by $I_\sigma$ the ideal generated by the irreducible component indexed by $\sigma$.
De Concini, Eisenbud, and Procesi proved that every nonzero $G$-invariant ideal is a
sum of such isotypic ideals \cite{deconcini1980young}.  Thus, after removing
containment-redundant indices, every nonzero proper $G$-invariant ideal can be written
in the form
\[
I_\Sigma=\sum_{\sigma\in\Sigma}I_\sigma,
\]
where $\Sigma\subset H_r$ is a finite nonempty antichain of Young diagrams.

When $\sigma=(t)$ consists of one row, $I_\sigma$ is the determinantal ideal
$D_t$ generated by the $t\times t$ minors of $X$.  More generally, if
$\sigma=(\sigma_1,\ldots,\sigma_s)$, then
\[
D_\sigma=D_{\sigma_1}\cdots D_{\sigma_s}
\]
is the product of the corresponding determinantal ideals.  De Concini,
Eisenbud, and Procesi determined primary decompositions, integral closures,
and symbolic powers for these products \cite{deconcini1980young}.  Whitehead
subsequently described products of arbitrary isotypic ideals through the
Littlewood--Richardson rule \cite{WhiteheadThesis}.  Nevertheless, explicit
formulas for the symbolic powers of arbitrary $G$-invariant ideals have not
previously been available. Our objective is to fill this gap.

A principal difficulty is that arbitrary isotypic ideals need not possess
bases of standard monomials.  Consequently, the standard monomial methods
that are effective for determinantal ideals and their products do not apply
directly.  Our approach instead combines Whitehead's product formula with
Littlewood--Richardson coefficients, Yamanouchi tableaux, and truncation
operations on Young diagrams.

If $P$ is a prime ideal of a polynomial ring over a perfect field and $m$ is a positive integer, then the contraction $P^mR_P\cap R$ --equivalently, the $P$-primary component in a minimal primary decomposition of $P^m$--is the set of functions vanishing to order at least $m$ at every closed point of the variety defined by $P$, a result of geometric significance due to Zariski and Nagata \cite{Zariski54}; see also the work of Eisenbud and Hochster \cite{EisenbudHochster}. Algebraically, this construction has several generalizations collectively called {\em symbolic power ideals}.
For $1\leq t\leq r$ and $m\geq1$, we define the generalized symbolic power of a $G$-invariant ideal $I$ with respect to $D_t$ by
\[
I^{(m)_t}=I^mR_{D_t}\cap R.
\]
Equivalently, this ideal has the saturation description
\[
I^{(m)_t}=I^m:D_{t-1}^{\infty},
\qquad D_0:=R.
\]
These ideals capture scheme structures defined by functions vanishing to high order on a determinantal variety, while allowing varying embedded components. As $t$ varies, they recover the two customary symbolic-power notions obtained from the associated and minimal primes. The case $t=2$ also includes saturation with respect to the homogeneous maximal ideal.

In this paper we obtain explicit combinatorial formulas for the generalized symbolic powers of $G$-invariant ideals in terms of Young diagrams and Littlewood--Richardson coefficients \Cref{lemma: Symbolic power one diagram invariant}, \Cref{prop: sigma'}. These descriptions underpin all our later results. The truncation formula in  \Cref{prop: sigma'} also leads naturally to the generalized symbolic Rees algebras
 \[ \mathcal R_s^t(I)=\bigoplus_{m\geq0} I^{(m)_t}T^m. \]
 In \Cref{thm: symbolic Rees Sigma} we prove that $\mathcal R_s^t(I)$ is a finitely generated, and hence Noetherian, $R$-algebra for every $G$-invariant ideal $I$.  This result is part of the broader circle of questions surrounding Hilbert's fourteenth problem, which asks when naturally occurring subalgebras of finitely generated algebras are themselves finitely generated.  In our setting, we prove finite generation by representation-theoretic methods: we pass to invariants under a maximal unipotent subgroup, use highest-weight theory  to encode the relevant representations by an affine semigroup, and then apply Grosshans' finite-generation theorem together with Gordan's lemma. This parallels a theorem of Herzog, Hibi, and Trung establishing finite generation of the saturated Rees algebra $\bigoplus_{m\geq0}(J^m:Q^\infty)T^m$ for any pair of monomial ideals $J$ and $Q$ \cite[Theorem~3.2]{herzog2007}. Thus the common polyhedral mechanism is finite generation of affine semigroups, with monomial exponent vectors in their setting replaced here by highest weights of representations. To our knowledge, this representation-theoretic approach has not previously been used to establish finite generation for symbolic Rees algebras of this kind. In particular, specializing our theorem to determinantal ideals recovers the Noetherianity of their symbolic Rees algebras and gives a new proof of that classical result.

Another way to study symbolic powers is by comparing them to the more tractable ordinary powers. This leads to the containment problem of determining the pairs of integers $(a,b)$ such that $I^{(a)}\subseteq I^b$. Swanson proved in \cite{Swanson} that when the ordinary and symbolic topologies are equivalent, the exponent $a$ can be bounded above by a linear function of $b$. Subsequent breakthroughs of Ein--Lazarsfeld--Smith \cite{ELS}, Hochster--Huneke \cite{comparison}, Ma--Schwede \cite{MaSchwede}, and Murayama \cite{Murayama2021UniformBounds} provide a uniform upper bound for the slope of this linear function for all ideals in a regular ring. For specific ideals, however, the containment problem can be refined further, and the linear equivalence of the two topologies can be quantified through asymptotic invariants. This perspective was pioneered by Bocci and Harbourne \cite{BocciHarbournecomparing, BocciHarbourneresurgence} through the Waldschmidt constant and resurgence of an ideal, and extended in \cite{guardo2013asymptotic} via the asymptotic resurgence. We introduce these invariants below.

The {\em Waldschmidt constant} of  a homogeneous ideal $I$ is the asymptotic growth rate of the least degree of a nonzero element in its symbolic powers given by
\[
\widehat{\alpha}(I)=\lim_{m\to\infty} \frac{\alpha(I^{(m)})}{m}.
\]
The {\em resurgence} and the {\em asymptotic resurgence} quantify the above mentioned linear equivalence and are given by
 \[
 \rho(I)=\sup\left \{\frac{a}{b} \mid I^{(a)}\not\subseteq I^b \right \} \text{ and } \widehat{\rho}(I)=\sup\left \{\frac{a}{b} \mid I^{(at)}\not\subseteq I^{bt} \text{ for }t\gg0 \right \}.
 \]
The connections between these invariants are best cast in terms of valuations, of which the least degree of a homogeneous ideal, denoted $\alpha(-)$ is one. We defer this perspective to \Cref{ss: asymptotic}.

The following theorem compiles our main results  from \Cref{lemma: Symbolic power one diagram invariant},
\Cref{prop: sigma'}, \Cref{theorem: waldschmidt constant one diagram}, \Cref{theorem: resurgence}, \Cref{thm: symbolic Rees Sigma}.

\begin{introthm*}
Let $\Sigma\subset H_r$ be a set of Young diagrams, set $I_\Sigma=\sum_{\s\in\Sigma}I_\s$, and let $1\leq t \leq r$ be an  integer.
\begin{enumerate}
\item For every $m\geq 1$ the generalized symbolic powers are described by
\[
I_\Sigma^{(m)_t}=\sum_{\l\in\Lambda_m(\Sigma)} I_{\l_{\geq t}} = \sum_{\lambda\in\Lambda_m(\Sigma_{\leftarrow t})} I_{\lambda_{\rightarrow t}},
\] where
\begin{itemize}
\item $\Lambda_m(\Sigma)$ is the set of Young diagrams $\l$ for which the iterated Littlewood--Richardson coefficient
$c_{{\s^{(1)}}^\vee,{\s^{(2)}}^\vee,\ldots, {\s^{(m)}}^\vee}^{\l^\vee}\neq 0$ for some $\s^{(1)},\s^{(2)},\ldots, \s^{(m)}\in \Sigma$,
\item $\Sigma_{\leftarrow t}$ is obtained from $\Sigma$ by removing the leftmost $t-1$ columns of each diagram and
\item $\lambda_{\rightarrow t}$ is obtained from $\l$ by adding $t-1$ columns, each having the same height as the first column of $\l$
\end{itemize}
\item  the following generalized symbolic Rees algebra  is Noetherian
 \[ \mathcal R_s^t(I_\Sigma)=\bigoplus_{m\geq0} I_{\Sigma}^{(m)_t}T^m, \]
\item  the Waldschmidt constant for the generalized symbolic powers is
\[
\widehat{\alpha}^t(I_\Sigma)=\lim_{m\to\infty}\frac{\alpha(I_\Sigma^{(m)_t})}{m}=\frac{r \ \gamma_{t}(\Sigma)}{r-t+1}
\]
\item if all the Rees valuations of $I_\Sigma$ measure $D_j$-adic order for $1\leq j\leq r$ and $\gamma_t(\Sigma)>0$, the asymptotic resurgence controlling the linear equivalence between the topologies of generalized symbolic powers and  ordinary powers  of $I_\Sigma$ is
\[
\widehat{\rho}(I_\Sigma^{(\bullet)_t}, I_\Sigma^\bullet)=\sup\left\{ \frac{a}{b} :  I_\Sigma^{(ak)_t}\not\subseteq I_\Sigma^{bk} \text{ for }k\gg0 \right\}=\frac{(r-t+1)\gamma_1(\Sigma)}{r\gamma_t(\Sigma)},
\]
where $\gamma_1(\Sigma)$ is the minimum number of boxes among the diagrams in $\Sigma$, while $\gamma_t(\Sigma)$ is the minimum number of boxes lying weakly to the right of column $t$ among the diagrams in $\Sigma$.
\end{enumerate}
\end{introthm*}

Note that parts (1)--(3) of the above theorem apply to all proper $G$-invariant ideals as each can be written in the form $I_\Sigma$ for some set $\Sigma$ of Young diagrams. In particular, parts (3) and (4) recover and vastly generalize the corresponding formulas $\widehat{\alpha}(D_t)=\frac{r}{r-t+1}$ and $\widehat{\rho}(D_t)=\frac{t(r-t+1)}{r}$ appearing for example in \cite[Example 3.15]{bdhmrational} and \cite[Theorem 4.1]{KumarMukundan}.
We also specialize these formulas to products of determinantal ideals
$D_\sigma$.  Even in this case the above results are new.

An ingredient of independent interest is a nonvanishing theorem for
iterated Littlewood--Richardson coefficients. This result supplies the rectangular diagrams
needed to attain the upper bounds in our computation of the skew Waldschmidt
constants. If $\sigma\in H_r$ has
$\ell=|\sigma|$ boxes, we prove in \Cref{thm: rectangle corrected} that the following  iterated Littlewood--Richardson coefficient is non-zero
\[
c_{\underbrace{\scriptstyle
\sigma^\vee,\ldots,\sigma^\vee}_{r\text{ copies}}}^{(\ell^r)}>0.
\]
Equivalently, this provides a concrete nonvanishing statement for corresponding powers of Schubert classes.   As such it may be of independent interest in Schubert calculus and intersection theory beyond its role in the proof of our asymptotic bounds.

The paper is organized as follows.  In \Cref{s: prelim} we review isotypic
ideals, Littlewood--Richardson coefficients, generalized symbolic powers,
Rees valuations, and asymptotic invariants.  In \Cref{s: symbolic} we
determine ordinary and generalized symbolic powers and establish the
truncation formula that makes these powers computationally accessible.  In
\Cref{s: asymptotic} we compute skew Waldschmidt constants and asymptotic
resurgence and then specialize the results to isotypic ideals and products of
determinantal ideals.  In \Cref{s: symbolic Rees} we prove finite generation
of the generalized symbolic Rees algebras.  The appendix proves the required
Littlewood--Richardson and Schubert-class nonvanishing theorem.

\begin{acknowledgement}
		This material is based upon work supported by the National Science Foundation under Grant No.\,DMS--1928930 and by the Alfred P. Sloan Foundation under grant G-2021-16778, while the second author was in residence at the Simons Laufer Mathematical Sciences Institute (formerly MSRI) in Berkeley, California, during the Spring 2024 semester. The second author is partially supported by NSF DMS--2401482. Macaulay2 \cite{M2} and ChatGPT-5.6 Luna were used to compute and verify examples relevant to this paper.
	\end{acknowledgement}

\section{Preliminaries}\label{s: prelim}

In this section we recall background material on Young diagrams, determinantal ideals, and
standard monomials that will be used throughout the paper. Our exposition follows
\cite{deconcini1980young, bruns2006determinantal, bruns2022determinants, FULTON_TABLEAUX}.

Throughout this paper, $\kk$ denotes a field of characteristic zero.

\subsection{Young diagrams, determinantal and isotypic ideals}

In this section we use notational conventions from \cite{deconcini1980young}.

Let $X=(x_{ij})$ be a generic $r\times n$ matrix of indeterminates, with $r\leq n$, and write $R=\kk[X]$ for the corresponding polynomial ring.
Regarding $X=(x_{ij})$ as a matrix, one can define an action of the group $G=\GL_r(\kk)\times \GL_n(\kk)$ on $R$ by the formula $(A,B)\cdot X=A^{-1}XB$. More precisely, the action of $(A,B)$ sends $x_{ij}$ to the $(i,j)$-entry of the matrix $A^{-1}XB$.

In this paper we study the ideals of $R$ that are invariant under the action of $G$. These ideals are often analyzed using methods from representation theory, which can be phrased in terms of Young diagrams.

\begin{definition}
A \emph{Young diagram} is a finite sequence $\sigma=(\sigma_1,\ldots,\sigma_s)$
of positive integers such that
$
\sigma_1 \ge \sigma_2 \ge \cdots \ge \sigma_s \ge 1.
$
We think of $\sigma$ as a sequence of rows of boxes of lengths $\sigma_1,\ldots,\sigma_s$. Thus $\sigma$ has $s$ rows and largest part $\sigma_1$, and we denote its size by $|\s|=\sum_{i=1}^s \s_i$.

The \emph{dual} (or \emph{conjugate}) partition $\sigma^\vee$ is defined by reflecting $\s$ across the main diagonal so that
\[
\sigma^\vee_j = \#\{ i \mid \sigma_i \ge j \}, \quad j = 1, 2, \dots, \sigma_1.
\]
For a positive integer $r$, we denote by $H_r$ the set of Young diagrams with at most $r$ columns, equivalently diagrams whose largest part
is at most $r$. For $\s\in H_r$, $\s^\vee$  has at most $r$ rows.
\begin{example}\label{ex: dual}
\phantom{a}
\begin{figure}[h]
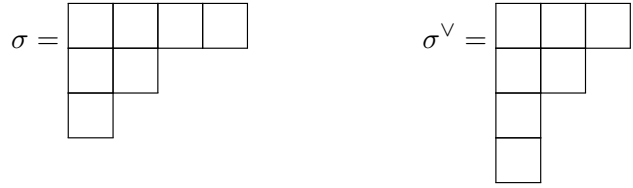

\centering
\[
\sigma=\ydiagram{4,2,1}
\qquad\qquad \qquad
\sigma^\vee=\ydiagram{3,2,1,1}
\]
\caption{A Young diagram $\sigma=(4,2,1)$  and its dual
diagram $\sigma^\vee=(3,2,1,1)$.}
\end{figure}
\end{example}
For Young diagrams $\lambda=(\lambda_1,\lambda_2,\ldots)$ and $\mu=(\mu_1,\mu_2,\ldots)$, we pad the shorter sequence with zeros and write $\mu\subseteq\lambda$ when $\mu_i\leq\lambda_i$ for every $i$. If $\mu\subseteq\lambda$, the \emph{skew Young diagram} $\lambda/\mu$ is the set
\[
\lambda/\mu=\{(i,j)\mid i\geq1,\ \mu_i<j\leq\lambda_i\},
\]
obtained by removing the Young diagram of $\mu$ from the top-left corner of the Young diagram of $\lambda$.
\end{definition}

\begin{definition}\label{def: standard}
If $\sigma$ is a Young diagram and $q$ is a positive integer, then a {\em tableau} $T$ of shape $\sigma$ on the alphabet $\{1,\ldots,q\}$ is a map $T:\sigma \to \{1,\ldots,q\}$, thought of as filling the boxes of $\sigma$ with numbers between 1 and $q$. We write $T(i,j)$ for the $j$-th number in row $i$ of this filling.

The {\em canonical tableau} of shape $\sigma$ is the  tableau $C_\sigma$ whose $i$-th row is $(1,2, \ldots, \sigma_i)$.

A tableau is called {\em standard}\footnote{The notions described here would be more aptly named semistandard tableau and semistandard bitableau because they have semistandard fillings. However, we use the word standard to agree with the terminology in \cite{deconcini1980young, bruns2006determinantal, bruns2022determinants}.}   if its rows are strictly increasing sequences when read from left to right and its columns are non-decreasing sequences when read from top to bottom.
	A {\em (standard) bitableau} is a pair $(S \mid T)$ of (standard) tableaux having the same shape, where $S$ is a tableau on $\{1,\ldots,r\}$ and $T$ is a tableau on $\{1,\ldots,n\}$.
Standard bitableaux can be used to indicate products of minors of a matrix: we associate to a bitableau $(S \mid T)$ of shape $\sigma$ the following product of minors, whose sizes are the parts $\sigma_i$:
 \[
 (S \mid T) \leftrightarrow \prod_i [S(i, 1), S(i, 2), \ldots,  S(i, \sigma_i)\mid  T(i, 1), T(i, 2), \ldots, T(i,\sigma_i)],
 \]
 where $[d_1,\ldots,d_k \mid j_1,\ldots,j_k]$ denotes the minor of $X$ in rows $1\le d_1<\cdots<d_k\le r$  and columns
$1\le j_1<\cdots<j_k\le n$. We henceforth conflate bitableaux and polynomials using this identification.
\end{definition}

Products of minors indexed by Young diagrams lead naturally to determinantal ideals.

\begin{notation}
For $1\le t\le r$, we denote by $D_t$ the \textit{determinantal ideal} of $R$ generated by the $t\times t$ minors of $X$.
The ideals $D_t$ are prime and form a nested chain
$D_r \subset D_{r-1} \subset \cdots \subset D_1$.
We also set $D_0=R$.

Let $\sigma=(\sigma_1,\ldots,\sigma_s)\in H_r$. We define
\begin{equation}\label{eq: Dsigma}
D_\sigma = D_{\sigma_1}D_{\sigma_2}\cdots D_{\sigma_s}.
\end{equation}
More generally, for a subset $\Sigma\subset H_r$ we set
\begin{equation}\label{eq: DSigma}
D_\Sigma=\sum_{\sigma\in\Sigma} D_\sigma.
\end{equation}
\end{notation}

In the bulk of our paper we are concerned with a more general family of $G$-invariant ideals, which we term isotypic ideals. To define them, we recall the representation-theoretic decomposition of $R$. Set $V=\kk^r$ and $W=\kk^n$. Since the action on the matrix of variables is $(A,B)X=A^{-1}XB$, the degree-one component of $R$ is the $G$-module $V^*\otimes_\kk W$. In characteristic zero, the Cauchy decomposition \cite[Corollary 2.3.3]{Weyman} gives
\[
R=\operatorname{Sym}(V^*\otimes_\kk W)
=\bigoplus_{\sigma\in H_r}M_\sigma,
\qquad
M_\sigma
=\mathbb S_{\sigma^\vee}(V^*)\otimes_\kk
  \mathbb S_{\sigma^\vee}(W).
\]
Thus the two factors of $M_\sigma$ are irreducible representations of $\GL_r(\kk)$ and $\GL_n(\kk)$, respectively. The first factor has a basis indexed by the standard tableaux $S$ of shape $\sigma$ on $\{1,\ldots,r\}$, and the second has a basis indexed by the standard tableaux $T$ of shape $\sigma$ on $\{1,\ldots,n\}$. Under the standard-bitableau realization, the map
\[
(S\mid C_\sigma)\otimes(C_\sigma\mid T)\longmapsto(S\mid T)
\]
identifies $M_\sigma$ with the span of the standard bitableaux of shape $\sigma$; see \cite[Corollary 3.4~(2)]{deconcini1980young}.

We next consider the ideal of $R$ generated by each isotypic component $M_\s$. By  \cite[Corollary 3.4~(2)]{deconcini1980young} this ideal is the same as the ideal $I_\s$ in \Cref{def: I_sigma}.

\begin{definition}\label{def: I_sigma}
For  $\sigma\in H_r$ let $K_\s$  be the product, over the parts $\sigma_i$, of the corresponding leading principal minors
\begin{equation}\label{eq: K}
K_\sigma:= \prod_{i=1}^s [1 \ 2 \ldots \sigma_i \mid 1 \ 2 \ldots \sigma_i].
\end{equation}
Define $I_\sigma$ to be the smallest G-invariant ideal containing $K_\sigma$.

Throughout the paper, when an ideal is written as $I_\Sigma$, the set $\Sigma\subset H_r$ is taken to be a finite nonempty antichain under diagram containment, and we set $I_\Sigma=\sum_{\s\in\Sigma}I_\s$. By Dickson's lemma any collection of diagrams may be replaced by the antichain of its containment-minimal elements without changing the resulting ideal.
\end{definition}

\begin{remark}\label{rem: G-invariant}
The following properties hold:
\begin{enumerate}
\item $I_\s=\bigoplus_{\tau\supseteq\s}M_\tau$ \cite[Theorem 4.1]{deconcini1980young};
\item A $G$-submodule $I$ of $R$ is an ideal if and only if $I=I_\Sigma=\sum_{ \s\in \Sigma} I_\s$ for some subset $\Sigma\subseteq H_r$ (this follows from (1) and \cite[Corollary 4.2]{deconcini1980young}).
\end{enumerate}
\end{remark}

A crucial difference between isotypic ideals  and the determinantal ideals discussed previously in \eqref{eq: DSigma} is that the former need not admit a $\kk$-basis of standard monomials. For example, it is shown in  \cite[Example 11.8.5]{bruns2022determinants} that $I_{(1,1)}$ does not have a standard monomial basis. However  the determinantal ideals $D_\sigma$ can be recovered as the integral closures $D_\sigma=\overline{I_\sigma}$ by \cite{deconcini1980young}.

\subsection{Littlewood-Richardson coefficients}\label{s: Schur}

In this section we use the conventions from \cite{FULTON_TABLEAUX}.

A \textit{semistandard Young tableau}\footnote{Note that this is dual to our use of {\em standard tableau} in \Cref{def: standard}. We shall use the terminology standard only in the sense of \Cref{def: standard} in this paper.} of shape $\lambda$ on the alphabet $\{1,\ldots,r\}$ is a filling of the boxes of $\lambda$ with entries in $\{1,\ldots,r\}$ that are weakly increasing from left to right along each row and strictly increasing from top to bottom along each column. We write $H_r^\vee$ for the set of Young diagrams having at most $r$ rows as these are the conjugates of elements in $H_r$ and ${\rm SSY}(\l)$ for the set of semistandard Young tableaux of shape $\l\in H_r^\vee$.

\begin{definition}
Let $\lambda\in H_r^\vee$. The \emph{Schur polynomial} in $r$ variables is
\[
s_\lambda(x_1,\ldots,x_r)=\sum_{T\in {\rm SSY}(\l)} x^T\in\kk[x_1,\ldots,x_r],
\]
where the sum ranges over all semistandard Young tableaux $T$ of shape $\lambda$ on $\{1,\ldots,r\}$ and
\[
x^T=\prod_{i=1}^r x_i^{\,\#\{\text{entries of $T$ equal to $i$}\}}.
\]
For $\mu,\nu\in H_r^\vee$, the product of Schur polynomials has the expansion
\begin{equation}\label{eq: Schur}
s_\mu(x_1,\ldots,x_r)s_\nu(x_1,\ldots,x_r)
=\sum_{\lambda\in H_r^\vee} c_{\mu\nu}^{\lambda}s_\lambda(x_1,\ldots,x_r),
\end{equation}
where the coefficients $c_{\mu\nu}^{\lambda}$ are the nonnegative integers called the \emph{Littlewood--Richardson coefficients}. More generally, for $\mu^{(1)},\ldots,\mu^{(q)}\in H_r^\vee$, we define the iterated Littlewood--Richardson coefficients by
\begin{equation}\label{eq: iterated Schur}
\prod_{i=1}^q s_{\mu^{(i)}}(x_1,\ldots,x_r)
=\sum_{\lambda\in H_r^\vee}
c_{\mu^{(1)}\mu^{(2)}\cdots\mu^{(q)}}^{\lambda}
s_\lambda(x_1,\ldots,x_r).
\end{equation}
\end{definition}

\begin{definition}
The \emph{row-reading word} of a tableau is the sequence obtained by reading its entries row by row, from left to right within each row, starting with the bottom row and proceeding upward to the top row. The \emph{column-reading word} of a tableau is the sequence obtained by reading its entries column by column, from bottom to top within each column, starting with the leftmost column and proceeding towards the rightmost.

A sequence of integers is called a {\em lattice word} or  \emph{Yamanouchi word} if for each $k$, the multiset consisting of the first $k$ entries contains at least as many $i$'s as $j$'s for all $i<j$. It is called a {\em reverse lattice word} if its reverse is a lattice word.

A (skew) tableau is called a \emph{Littlewood--Richardson tableau} if it is semistandard and its row-reading word is a reverse lattice word.

A (skew) tableau  is said to be \emph{Yamanouchi} if it is standard in the sense of \Cref{def: standard} and  the {\em column-reading word} of the tableau is a Yamanouchi word.
\end{definition}

The Littlewood--Richardson coefficients $c_{\mu \nu}^{\lambda}$ can  be interpreted combinatorially as the number of Littlewood--Richardson skew tableaux of shape $\lambda / \mu$ and content $\nu$ \cite[Proposition 3, p.~64]{FULTON_TABLEAUX}. This description is known as the \emph{Littlewood--Richardson rule}.

\begin{example}
The row-reading word of the skew tableau shown below on the left is $3,\,2,\,3,\,1,\,2,\,1$. It is a reverse lattice word.
The column-reading word of the skew tableau shown below on the right is $1,\,2,\,1,\,3,\,2,\,3$. It is  a Yamanouchi word.
\begin{figure}[h]
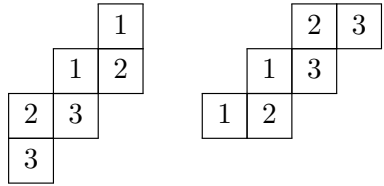

\centering
\[
\begin{ytableau}
\none & \none & 1  \\
\none & 1 & 2 \\
 2 & 3  \\
3
\end{ytableau}
\qquad
\begin{ytableau}
\none & \none & 2 & 3   \\
\none & 1 & 3  \\
1 & 2 \\
\end{ytableau}
\]

\caption{A Littlewood-Richardson skew tableau of shape $(3,3,2,1)/(2,1)$ on the left and a Yamanouchi skew tableau of shape $(4,3,2)/(2,1)$ on the right.}\label{fig 3}
\end{figure}
\end{example}

We will use a dual interpretation of Littlewood--Richardson coefficients in terms of Yamanouchi tableaux in \Cref{lemma Lambda=littlewood rich}.
The notions presented above are dual in the following sense: under dualization of tableaux as in \Cref{ex: dual}, row-increasing tableaux become column-increasing tableaux, column-reading words become the reverse of row-reading words, and the Yamanouchi condition becomes the Littlewood--Richardson condition. That is, a (skew) tableau $\l$ is Yamanouchi if and only if its dual $\l^\vee$ is Littlewood--Richardson. The skew tableaux in \Cref{fig 3} are dual to each other.

\subsection{Symbolic powers and Rees valuations}

Since the determinantal ideals $D_t$ with $1\leq t\leq r$ are the only prime $G$-invariant ideals \cite[Theorem 5.2]{deconcini1980young} and they form the chain $D_r\subseteq D_{r-1}\subseteq\cdots\subseteq D_2\subseteq D_1$, all $G$-invariant ideals can be interpreted as scheme structures  on a variety of matrices of rank $<t$ with embedded components. In this section we study  controlled fattenings of these scheme structures. Algebraically these are determined by symbolic power ideals.

For an ideal $I$ of a commutative ring $R$, we use the conventions
\[
\Ass(I):=\Ass_R(R/I),\qquad \Min(I):=\Min_R(R/I).
\]
Two conventions for the $m$-th {\em symbolic power} are common:
\begin{equation}\label{eq: symbolic_power}
I_{\Ass}^{(m)}=\bigcap_{P\in \Ass(I)} \left(I^m R_P\cap R \right) \quad \text{or} \quad I_{\Min}^{(m)}=\bigcap_{P\in \Min(I)} I^m R_P\cap R.
\end{equation}
To study both versions of symbolic powers simultaneously we introduce the following generalization.

\begin{definition}\label{def: symbolic_power}
Let $I$ be a $G$-invariant ideal and let $m\geq 1$ and $1\leq t\leq r$ be integers. We define the $m$-th {\em symbolic power of $I$ with respect to $D_t$} as follows
\[
I^{(m)_t}= I^m R_{D_t}\cap R.
\]
\end{definition}
We will see in \Cref{lem: symbolic_powers_maxass} that \Cref{def: symbolic_power} recovers $I_{\Ass}^{(m)}$ and $I_{\Min}^{(m)}$ in  \eqref{eq: symbolic_power} when applied for the smallest and the largest values of $t$ such that $D_t\in \Ass(I)$, respectively, and that it can also be viewed as the saturation of $I^m$ with respect to $D_{t-1}$. From this point of view one can interpret the relative notion of symbolic powers in \Cref{def: symbolic_power} as a bridge between the two more common notions of symbolic powers from the literature as well as a generalization of the saturation with respect to the homogeneous maximal ideal (the case $t=2$).

To study the symbolic powers of determinantal ideals $D_t$ one  defines valuations on the quotient field ${\rm Frac}(R)$ of $R$ as follows:
\begin{equation}\label{eq: nu}
\gamma_{t}(x)=\max \left \{i \geq 0: x\in D_t^{(i)} \right \} \text{ for } 0\neq x\in R.
\end{equation}
The function $\gamma_t$ is a valuation since the associated graded ring of $R_{D_t}$ at its maximal ideal is a domain \cite[\S9--10]{bruns2006determinantal}. To be able to effectively compute the valuations in \eqref{eq: nu} one utilizes the theory of standard monomials and the following combinatorial function on Young diagrams.

\begin{definition}
Let $\sigma=(\sigma_1,\ldots,\sigma_s)\in H_r$ and let $t\in \mathbb{N}$. We define the function $\gamma_t:H_r\to \NN$ by
\begin{equation}\label{eq: gamma}
\gamma_t(\sigma)=\sum_{i=1}^s \max\{0,\sigma_i-t+1\}.
\end{equation}
The function $\gamma_t(\sigma)$ counts the number of boxes of $\sigma$ lying weakly to the
right of the $t$-th column. For a finite nonempty set $\Sigma\subset H_r$, we define
\[
\gamma_t(\Sigma):=\min\{\gamma_t(\sigma):\sigma\in\Sigma\}.
\]
More generally, for a valuation $v$ and a nonzero ideal $J$, we employ the convention $v(J):=\min\{v(f):0\neq f\in J\}$. Since $\gamma_t(I_\s)=\gamma_t(\s)$ by \cite[Theorems~4.1 and 7.1(1)]{deconcini1980young} (see also \Cref{lemma: Symbolic power one diagram invariant}~(3)), we have  $\gamma_t(I_\Sigma)=\gamma_t(\Sigma)$.
\end{definition}

\begin{example}
\phantom{a}
\begin{figure}[h]
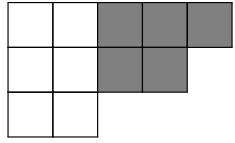

\centering
\begin{ytableau}
  \, & \, & *(gray) & *(gray) & *(gray)\\
  \, & \, &  *(gray) & *(gray)\\
  \, & \, \\
\end{ytableau}
\caption{The shaded boxes contribute to $\gamma_3(5,4,2)=5$.}
\end{figure}
\end{example}

The valuations $\gamma_t$ in \eqref{eq: nu} can be extended to the polynomial ring $R[T]$ and yield a representation of the Rees algebra of $D_t$  as an intersection of $R[T]$ with finitely many valuation rings. This idea extends to the notion of Rees valuations for arbitrary ideals.

\begin{def-prop}\label{def: Rees_valuations}
Let $R$ be a Noetherian domain and $I$ an ideal of $R$. Then there exist finitely many discrete valuation rings $V_1 ,\ldots , V_s$ of rank one with $R\subset V_i \subset {\rm Frac}(R)$ such that
\begin{enumerate}
\item For all $n\in \NN$, $\overline{I^n}=\bigcap_{i=1}^s (I^nV_i\cap R)$;
\item The set $\{V_1,\ldots ,V_s\}$ satisfying (1) is minimal possible.
\end{enumerate}
The  valuations corresponding to $V_1, \ldots, V_s$ are unique and they are called the {\em Rees valuations} of $I$.
\end{def-prop}

As pioneered by Kaveh--Khovanskii \cite{KAVEH_KHOVANSKII} and Lazarsfeld--Musta\c{t}\u{a} \cite{lazarsfeld09}, it is fruitful to utilize valuations to construct convex bodies in Euclidean space called Newton-Okounkov bodies.
Using the functions $\gamma_t$ as coordinates one can define a map $\gamma=(\gamma_1, \ldots, \gamma_r):H_r\to \NN^r$. These functions can be used to encode collections of Young diagrams as convex sets in $\RR^r$ using the following method.

\begin{notation}\label{not: P_sigma}
If $\Sigma\subset H_r$ is a collection of Young diagrams, we define a convex set
\[
P_\Sigma=\operatorname{conv}\bigl\{(\gamma_1(\sigma),\gamma_2(\sigma),\ldots,\gamma_r(\sigma))
\mid \sigma\in\Sigma\bigr\}\subset \mathbb{R}^r.
\]
\end{notation}

If $\Sigma$ is finite, then $P_\Sigma$ is a polytope. This polytope captures information about the integral closure of the $G$-invariant ideal associated to $\Sigma$.
By  \cite[Theorem 8.1]{deconcini1980young},
 the integral closure of $I_\s$ is $\overline{I_\s}=D_\s$ and by  \cite[Theorem 8.2]{deconcini1980young} (see also \cite[Theorem 2.3]{HV})
 the integral closure of the $G$-invariant ideal $I_\Sigma=\sum_{ \s\in \Sigma} I_\s$ is computed as follows:
\begin{equation}\label{eq: integral closure}
\overline{I_\Sigma}
=\overline{\sum_{\sigma\in\Sigma}I_\sigma}
=\overline{\sum_{\sigma\in\Sigma}D_\sigma}
=\sum_{\substack{\tau\in H_r\\
\gamma(\tau)\in P_\Sigma+\RR_{\geq0}^r}}I_\tau.
\end{equation}
Thus the indexing condition in the final sum is membership of the valuation vector $\gamma(\tau)$ in the Newton polyhedron $P_\Sigma+\RR_{\geq0}^r$. Note that taking integral closure of the sum of the $D_\sigma$ is necessary, since such a sum need not be integrally closed.

The next result describes the Rees valuations of determinantal ideals in terms of these polytopes.

\begin{theorem}[{\cite[Theorem 3.13]{bdhmrational}}]\label{thm: Rees_valuations}
Let $\Sigma\subset H_r$ be finite and nonempty, and let $D_\Sigma=\sum_{\sigma\in\Sigma}D_\sigma$. For each non-coordinate facet $F$ of
$
P_\Sigma+\RR_{\geq0}^r
$
whose supporting hyperplane has positive constant term, write its equation uniquely as
$
h_1x_1+\cdots+h_rx_r=c,
$
where $h=(h_1,\ldots,h_r)\in\ZZ_{\geq0}^r$ is the primitive inward normal and $c\in\ZZ_{>0}$. Then the valuation associated to $F$ is
\[
v_F=\sum_{i=1}^r h_i\gamma_i,
\]
and the Rees valuations of $D_\Sigma$  are the irredundant valuations arising in this way.  \end{theorem}

We refer to a hyperplane as a coordinate hyperplane if it has the equation $x_i=0$. We refer to a valuation as a {\em coordinate valuation} if it is equal to $\gamma_i$ for some $i$. Thus coordinate valuations can correspond to non-coordinate hyperplanes $x_i=c>0$.

\subsection{Asymptotic invariants of symbolic powers}\label{ss: asymptotic}

 We define the resurgence and asymptotic resurgence numbers associated to a pair of graded families of ideals. These notions generalize the well-studied resurgence and asymptotic resurgence of an ideal in a polynomial ring as presented in  \Cref{s: intro}. A collection $\a_\bullet = \{\a_i\}_{i\geq 1}$ of ideals in $R$ is called a graded family if $\a_i\a_j \subseteq \a_{i+j}$ for all $i,j \geq 1$. A graded family $\a_\bullet$ is called a filtration if $\a_i \supseteq \a_{i+1}$ for all $i \geq1$.

 \begin{definition}
 Let $\a_\bullet$ and $\b_\bullet$ be graded families of ideals of $R$. Define the generalized {\em resurgence} and {\em asymptotic resurgence} for the given pair of families as
\[
 \rho(\a_\bullet, \b_\bullet)=\sup\left \{\frac{u}{v} \mid u,v\in\mathbb Z_{>0}, \a_u\not\subseteq \b_v\right \} \text{ and } \widehat{\rho}(\a_\bullet,\b_\bullet)=\sup\left \{\frac{u}{v} \mid u,v\in\mathbb Z_{>0}, \a_{ut} \not\subseteq \b_{vt} \text{ for }t\gg0 \right \}.
 \]
 \end{definition}
We will be interested in situations where $\a_\bullet$ is the family of generalized symbolic powers defined in \Cref{def: symbolic_power} and $\b_\bullet$ is the family of ordinary powers,  computing $ \widehat{\rho}(I^{(\bullet)_t},I^\bullet)$ for a $G$-invariant ideal $I$.

We recall for future use a valuation-theoretic description of asymptotic resurgence. In the case where $\a_\bullet$ consists of the classical symbolic powers, the theorem below was first discovered by DiPasquale--Francisco--Mermin--Schweig in \cite[Theorem~4.10]{dipasquale2019asymptotic}. We use the following special case of  a more encompassing result due to H\`a-Kumar-Nguyen-Nguyen in \cite{HKNN}.

\begin{theorem}[{\cite[Theorem 2.13 and Corollary 4.11]{HKNN}}]\label{thm: DFMS}
Let $\ba_\bullet$ be a filtration of nonzero ideals. Let $I$ be an ideal, $\b_\bullet$ given by $\b_w=I^w$ be the family of ordinary powers of $I$, and let $v_1,\ldots,v_s$ be the Rees valuations of $I$. Then
the asymptotic resurgence of $\a_\bullet$ with respect to both the family of powers of $I$, $I^\bullet=\{I^m\}_{m\geq 1}$, and their integral closures $\overline{I^\bullet}=\{\overline{I^m}\}_{m\geq 1}$  is given by
\[
\widehat{\rho}(\a_\bullet, I^\bullet)=\widehat{\rho}(\a_\bullet, \overline{I^\bullet})={\rho}(\a_\bullet, \overline{I^\bullet})=\max_{1\le i\le s}
\left\{\frac{v_i(I)}{\widehat{v}_i(\a_\bullet)}\right\},
\]
where
\[
\widehat{v}_i(\a_\bullet)=\lim_{m\to\infty}\frac{v_i(\a_m)}{m}.
\]
\end{theorem}

Because of the above considerations it becomes important to develop methods for determining the asymptotic behavior of valuations as in the preceding displayed limit. The most common valuation of a homogeneous ideal is given by the order defined by the powers of the homogeneous maximal ideal, equivalently for the standard grading by the smallest degree of  its elements. Considering this valuation leads to another important asymptotic invariant.

\begin{definition}
\label{def:waldschmidt}
Denote by $\alpha(I)$ the smallest degree of a nonzero homogeneous form in a homogeneous ideal $I$. The {\em Waldschmidt constant} of a graded family $\a_\bullet$ is the real number
\[
\widehat{\alpha}(\a_\bullet)=\lim_{m\to\infty}\frac{\alpha(\a_m)}{m}=\inf_{m\geq1}\frac{\alpha(\a_m)}{m}.
\]
Similarly, if $v$ is any valuation, one calls the limit below a {\em skew Waldschmidt constant} of $\a_\bullet$:
\[
\widehat{v}(\a_\bullet)=\lim_{m\to\infty}\frac{v(\a_m)}{m}=\inf_{m\geq1}\frac{v(\a_m)}{m}.
\]
The existence of these limits follows from subadditivity and Fekete's lemma.
\end{definition}

\section{Symbolic Powers of GL-Invariant Ideals}\label{s: symbolic}

\subsection{Generalized symbolic powers}

In this section we show how the notion of generalized symbolic powers from \Cref{def: symbolic_power}  encompasses the ideals typically called symbolic powers.

For a proper ideal $I$ of $R$ we denote by $\maxass(I)$ the set of maximal elements under containment of the set of associated primes $\Ass(I)$. Our interest in these primes is spurred by the following connection with  symbolic powers.

\begin{lemma}\label{lem: symbolic_powers_maxass}
For any positive integer $m$, the corresponding symbolic power of $I$ is described by
\begin{equation}\label{eq: maxass}
 I_{\Ass}^{(m)}=\bigcap_{P\in \maxass(I)} I^mR_P\cap R.
 \end{equation}
In particular, if $I$ is $G$-invariant and $t_0=\min\{t\mid D_t\in \Ass(I)\}$ and $t_1=\max\{t \mid D_t\in \Ass(I)\}$, then with notation as in \Cref{def: symbolic_power} we have
\[
 I_{\Ass}^{(m)}= I^{(m)_{t_0}}, \quad   I_{\Min}^{(m)}= I^{(m)_{t_1}} \quad \text{and} \quad I^{(m)_t}=I^m:D_{t-1}^{\infty}.
 \]
\end{lemma}
\begin{proof}
In view of \eqref{eq: symbolic_power}, the left-hand side of \eqref{eq: maxass} is contained in the right. However, for each $Q\in \Ass(I)$ there exists $P\in\maxass(I)$ such that $Q\subseteq P$ and consequently we obtain containments
 \[
 I^mR_P\cap R =\{f\in R : sf\in I^m \text{ for some } s\not \in P\} \subseteq \{f\in R : sf\in I^m \text{ for some } s\not \in Q\} =I^mR_Q\cap R.
 \]
 This proves that the right hand side of \eqref{eq: maxass} is contained in the left, thus finishing the proof.

 If $I$ is $G$-invariant then so are the associated primes of $I$. Thus $\Ass(I)\subseteq \{D_t\mid 1\leq t\leq r\}$ and in particular $\maxass(I)=\{D_{t_0}\} $ and  $\Min(I)=\{D_{t_1}\}$ since the prime determinantal ideals form a chain. Comparing \Cref{def: symbolic_power} with \eqref{eq: maxass}, \eqref{eq: symbolic_power} and the definition of saturation yields the last claim.
\end{proof}

Like the usual symbolic powers, our generalized symbolic powers form a graded family of ideals which is also a filtration.

\begin{lemma}\label{lem: generalized symbolic powers graded}
If $I$ is $G$-invariant, for  $1\leq t\leq r$ the sequence of ideals $I^{(\bullet)_t}=\{I^{(m)_t}\}_{m\geq 1}$ is a graded family and a filtration.
\end{lemma}
\begin{proof}
Since $I^i\supseteq I^j$ for $i\leq j$, we have $I^iR_{D_t}\cap R\supseteq I^j R_{D_t} \cap R$ and the filtration property follows. Since $I^iR_{D_t}\cdot I^j R_{D_t}=I^{i+j}R_{D_t}$, we have $(I^iR_{D_t}\cap R)\cdot (I^j R_{D_t} \cap R)\subseteq I^{i+j}R_{D_t}\cap R$ and the graded family property is established.
\end{proof}

\subsection{Powers of isotypic ideals and their associated primes}

The computation of arbitrary products $I_\sigma I_\tau$ of isotypic ideals  for  $\sigma, \tau\in H_r$, left open  in \cite{deconcini1980young},   was carried out by Whitehead in \cite{WhiteheadThesis} using the Littlewood--Richardson rule. To describe her work we need to introduce additional terminology.

\begin{notation}\label{def: D}
Let $\sigma$ and $\tau$ be Young diagrams.
We denote by $D(\sigma,\tau)$ the set of all partitions $\lambda$ for which there exists a skew tableau of shape $\lambda/\sigma$ with content
$
1^{\tau^\vee_1}2^{\tau^\vee_2}\cdots r^{\tau^\vee_r}
$
whose column-reading word is Yamanouchi. This set can be expressed in terms of Littlewood--Richardson coefficients by  \Cref{lemma Lambda=littlewood rich}, which states that
\[
D(\sigma,\tau)
=\{\lambda \mid c^{\lambda^\vee}_{\sigma^\vee, \tau^\vee}\neq 0\}.
\]
\end{notation}

We now recall Whitehead's formula for products of isotypic ideals.

\begin{theorem}[{\cite[Theorem 7.2]{WhiteheadThesis}}]\label{theorem: product whitehead}
Let $\sigma,\tau\in H_r$. Then
\[
I_\sigma I_\tau
=\sum_{\lambda\in D(\sigma,\tau)\cap H_r}I_\lambda.
\]
\end{theorem}

In the remainder of this section we aim to give a closed formula for the symbolic powers of isotypic ideals in the spirit of \Cref{theorem: product whitehead}. To do so it is necessary to state a few rules concerning the arithmetic of these ideals. These generalize \cite[Lemmas~7.4 and 7.5]{WhiteheadThesis}.

\begin{lemma}\label{lemma: sum and intersection}
\begin{enumerate}
\item
Let $\rho$, $\lambda_1,\dots,\lambda_u$, and $\mu_1,\dots,\mu_v$ be partitions in $H_r$. Then
\[
I_\rho+(I_{\lambda_1}+\cdots+I_{\lambda_u})\cap(I_{\mu_1}+\cdots+I_{\mu_v})
=
(I_\rho+I_{\lambda_1}+\cdots+I_{\lambda_u})\cap(I_\rho+I_{\mu_1}+\cdots+I_{\mu_v}).
\]
\item
Let $\rho$, $\lambda_1,\dots,\lambda_v$, and $\mu_1,\dots,\mu_v$ be partitions in $H_r$. Then
\[
(I_{\lambda_1}\cap I_{\mu_1})+\cdots+(I_{\lambda_v}\cap I_{\mu_v})
=(I_{\lambda_1}+\cdots+I_{\lambda_v})\cap J_1\cap\cdots\cap J_v,
\]
for some $\GL$-invariant ideals $J_i$ such that
 $I_{\mu_i}\subseteq J_i$ for each $1\leq i\leq v$.
\end{enumerate}
\end{lemma}

\begin{proof}
(1) The containment $\subseteq$ is immediate.
For the reverse containment, note that both sides are $G$-invariant ideals and hence decompose as sums of $I_\sigma$.
Using \cite[Lemma~7.4]{WhiteheadThesis}, which states that for arbitrary Young diagrams $\rho, \lambda, \nu$ the containment $I_\rho\subseteq I_\l+I_\nu$ implies $I_\rho\subseteq I_\l$ or $I_\rho\subseteq I_\nu$, containment in both summands on the right hand side of (1) forces containment in the left-hand side.

(2) The proof is by induction on $v$.
The case $v=2$ follows from repeated application of \cite[Lemma~7.5]{WhiteheadThesis}, which states that $I_\rho+I_\l\cap I_\nu=(I_\rho+I_\l)\cap (I_\rho+I_\nu)$.
We  first show
		\begin{equation}\label{eq: v=2}
			(I_{\l_1}\cap I_{\mu_1})+ (I_{\l_2}\cap I_{\mu_2})=(I_{\l_1}+I_{\l_2}\cap I_{\mu_2})\cap  (I_{\mu_1}+ I_{\l_2}\cap I_{\mu_2}).
		\end{equation}
		The left hand side is clearly contained in the right. Since the right hand side is a $\GL$-invariant ideal it can be written for some set $\Sigma$ of partitions as
		\[
		(I_{\l_1}+I_{\l_2}\cap I_{\mu_2})\cap  (I_{\mu_1}+ I_{\l_2}\cap I_{\mu_2})=\sum_{\sigma\in \Sigma} I_\sigma.
		\]
		It follows that for each $\s\in \Sigma$ we have $I_\s\subseteq I_{\l_1}+I_{\l_2}\cap I_{\mu_2}$ and $I_\s\subseteq I_{\mu_1}+ I_{\l_2}\cap I_{\mu_2}$. Utilizing the same argument as in \cite[Lemma  7.4]{WhiteheadThesis}, we obtain from the first containment that $I_\s\subseteq I_{\l_1}$ or $I_\s\subseteq I_{\l_2}\cap I_{\mu_2}$ and from the second containment that $I_\s\subseteq I_{\mu_1}$ or $I_\s\subseteq I_{\l_2}\cap I_{\mu_2}$. Overall, $I_\s\subseteq I_{\l_1}\cap I_{\mu_1} $ or $I_\s\subseteq I_{\l_2}\cap I_{\mu_2}$, which shows the reverse containment proving  \eqref{eq: v=2}.

		Denoting $J_1=I_{\mu_1}+ I_{\l_2}\cap I_{\mu_2}$ so that $I_{\mu_1}\subseteq J_1$ and  $J_2=I_{\l_1}+I_{\mu_2}$ so that $I_{\mu_2}\subseteq J_2$ we now rewrite \eqref{eq: v=2} as
		\[
		(I_{\l_1}\cap I_{\mu_1})+ (I_{\l_2}\cap I_{\mu_2})=(I_{\l_1}+I_{\l_2})\cap (I_{\l_1}+I_{\mu_2})\cap J_1.
		\]

		The case $v>2$ follows a similar pattern. By the inductive hypothesis we have
		\[
		(I_{\l_1}\cap I_{\mu_1})+ (I_{\l_2}\cap I_{\mu_2})+ \cdots +(I_{\l_v}\cap I_{\mu_v}) =(I_{\l_1}\cap I_{\mu_1})+(I_{\l_2}+\cdots+I_{\l_v})\cap J'_2\cap \cdots \cap J'_{v}.
		\]
		We claim
		\begin{equation}\label{eq: general v}
			(I_{\l_1}\cap I_{\mu_1})+(I_{\l_2}+\cdots+I_{\l_v})\cap J'_2\cap \cdots \cap J'_{v}\\
			=\left( I_{\l_1}+ (I_{\l_2}+\cdots+I_{\l_v})\cap J'_2\cap \cdots \cap J'_{v}\right )\cap J_1.
		\end{equation}
		where $J_1= I_{\mu_1}+ (I_{\l_2}+\cdots+I_{\l_v})\cap J'_2\cap \cdots \cap J'_{v}$.
		Moreover notice that by part (1) we can write
		\[
		I_{\l_1}+ (I_{\l_2}+\cdots+I_{\l_v})\cap J'_2\cap \cdots \cap J'_{v}=(I_{\l_1}+ I_{\l_2}+\cdots+I_{\l_v})\cap (I_{\l_1}+J'_2)\cap \cdots \cap (I_{\l_1}+J'_v).
		\]
		Setting $J_i=I_{\l_1}+J'_i$ for $2\leq i\leq v$ and utilizing the inductive hypothesis $I_{\mu_i}\subset J'_i$ yields $I_{\mu_i}\subset J_i$. Thus an equivalent formulation of  \eqref{eq: general v} is
		\begin{equation}\label{eq: general v pt 2}
			(I_{\l_1}\cap I_{\mu_1})+(I_{\l_2}+\cdots+I_{\l_v})\cap J'_2\cap \cdots \cap J'_{v}\\
			=(I_{\l_1}+ I_{\l_2}+\cdots+I_{\l_v})\cap (I_{\l_1}+J'_2)\cap \cdots \cap (I_{\l_1}+J'_v)\cap J_1.
		\end{equation}
		The left hand side of \eqref{eq: general v pt 2} is clearly contained in the right. The right hand side is a $\GL$-invariant ideal which we write as $\sum_{\sigma\in \Sigma} I_\sigma$ for some set $\Sigma$ of diagrams. Then for every $\s\in \Sigma$ it follows by  \cite[Lemma  7.4]{WhiteheadThesis}  that we have $I_\s\subset I_{\l_1}$ or ($I_\s\subset I_{\l_2}+\cdots +I_{\l_v}$ and $I_\s\subset J'_2, \ldots ,I_\s\subset J'_v$) and $I_\s\subset J_1$. Overall this means that $I_\s$ is contained in the left hand side of \eqref{eq: general v pt 2}, completing the proof.
		\end{proof}

Next we discuss the associated primes and primary decompositions of $G$-invariant ideals.

In view of \Cref{def: symbolic_power}, for a fixed integer $p$ it will be useful to separate the primary components with radical $D_t$ with  $t< p$ from those with radical  $D_t$ with  $t\geq p$. We now discuss the combinatorial principles behind such a decomposition.

\begin{notation}\label{def: <v}
			For a Young diagram $\l$ and an integer $p>0$  we set
					$\lambda_{\geqslant p}$ to be the diagram formed by those rows of $\lambda$ of length at least $p$	, equivalently, $\lambda_{\geqslant p}$ is obtained from $\l$ by removing all rows of length $<p$.  Set $\lambda_{<p}$ to be the partition obtained by truncating the rows of $\l$ to have size at most the size of the largest part of $\l$ that is less than $p$. We use the convention that the maximum of the empty set is $0$; thus, if $\lambda$ has no part smaller than $p$, then $\lambda_{<p}=\varnothing$. In detail,
					\[
					(\l_{\geq p})_i= \l_i \text{ if } \l_i\geq p, \ 0 \text{ otherwise},  \quad \text{and} \quad
					(\l_{<p})_i=\min\{\l_i, \max\{\l_j: \l_j<p\}\},
					\qquad \max\varnothing:=0.
					\]
	\end{notation}

	We now point out how the ideals $I_\l$, $I_{\l_{<p}}$ and $I_{\l_{\geqslant p}}$ relate.
\begin{lemma}\label{lem: <v}
For any partition $\lambda$ and positive integer $p$ there is a decomposition
$
I_\lambda=I_{\lambda_{\ge p}}\cap I_{\lambda_{<p}}.
$
Moreover $\Ass (I_{\lambda_{\ge p}})=\{D_{\l_i}: \l_i\geq p \}$ and $\Ass (I_{\lambda_{< p}})=\{D_{\l_i}: \l_i < p\}$.
\end{lemma}
\begin{proof}
By \cite[Cor.~5.4]{deconcini1980young},
\begin{equation}\label{eq: ass Is}
I_\lambda=\bigcap_i I_{\lambda^{(i)}},
\end{equation}
where $\lambda^{(i)}$ is the largest rectangular partition with parts equal to $\lambda_i$ contained in $\lambda$.
Applying the same result, the intersection of ideals corresponding to those rectangles with number of columns at least $p$ gives $I_{\lambda_{\ge p}}$, and intersecting the remaining ones gives $I_{\lambda_{<p}}$. Moreover, by \cite[Cor.~5.3]{deconcini1980young}, each of the ideals  $I_{\lambda^{(i)}}$ is $D_{\l_i}$-primary. This establishes the statement regarding associated primes.
\end{proof}

The following lemma gives the minimal and maximal associated primes of arbitrary $G$-invariant ideals. Recall from \Cref{rem: G-invariant}~(2) that these have the form $I_\Sigma$ with $\Sigma\subseteq H_r$. While the resulting decomposition may not be irredundant it suffices to establish the maximal associated prime of the ideal, an ingredient we use in our computation of symbolic powers.

\begin{lemma}\label{lemma: symbolic power for sum of ideal invariant}
		If $\Sigma \subset H_r$ is an antichain and $I_\Sigma=\sum _{\sigma \in \Sigma}I_\sigma$ then $\Min(I_\Sigma)=\{D_q\}$ and
$ \maxass(I_\Sigma) = \{D_p\}$,  where  $q=\min\{\s_1: \s\in \Sigma\}$   is the  length of  the shortest first row in any $\sigma \in \Sigma$ and
$p=\min\{\s_i: \s\in \Sigma\}$  is the  length of  the shortest (nonzero) row in any $\sigma \in \Sigma$.
	\end{lemma}
	\begin{proof}
	From \eqref{eq: ass Is} we see that for any diagram $\s$, $\Ass(I_\s)=\{D_{\s_i}\}$ and in particular $\Min(I_\s)=\{D_{\s_1}\}$. Thus with the notation introduced above we have $I_\s\subseteq D_q$ for all $\s\in \Sigma$ and thus $I_\Sigma\subseteq D_q$. Moreover by the definition of $q$, if $q=\s_1$ then $I_\s\not \subseteq D_{q+1}$ and thus $I_\Sigma\not \subseteq D_{q+1}$. It follows that the minimal associated prime of $I_\Sigma$ is $D_q$.

	To determine the maximal associated prime, the proof is by induction on  the total number of rows of shapes in $\Sigma$.
		Let $\tau \in \Sigma$ be an element  with a row of length $p$ having minimal number $s$ of rows among such elements. Write $\tau=\tau^{\prime}\sqcup (p)$ where $(p)$ represents the diagram with one row of length $p$ and $\tau^{\prime}$ consists of the remaining rows of $\tau$. Let $\rho=(p^s)$ be the rectangular diagram with $s$ rows of  length $p$, and let $\Sigma^\prime =\Sigma \setminus \{\tau\}$. If $\tau=\rho$ then $I_\Sigma$ is $D_p$-primary by \cite[Corollary 5.3]{deconcini1980young} as $\rho$ is rectangular with rows of length $p$ and no diagram in $\Sigma$ has any row of length $<p$. Moreover we have $q=p$ as evidenced by $\rho\in \Sigma$. This yields the desired conclusion.

		From now on we assume $\tau\neq \rho$.
		Then by \cite[p.~153 last paragraph]{deconcini1980young} there is a decomposition
		$$I_\Sigma= \left( I_\rho+\sum_{\sigma \in \Sigma^\prime} I_\sigma \right) \cap \left(I_{\tau^\prime} +\sum_{\sigma\in \Sigma^\prime}I_\sigma  \right).$$
		Set  $Q:= I_\rho+\sum_{\sigma \in \Sigma^\prime} I_\sigma $ and $J := I_{\tau^\prime} +\sum_{\sigma\in \Sigma^\prime}I_\sigma $.
		We show that $J$ is not contained in $Q$, thus $Q$ is irredundant in the decomposition above. Indeed $I_{\tau'}\not\subseteq I_\rho$ as $\rho\not \subseteq \tau'$ and  $I_{\tau'}\not\subseteq \sum_{\sigma \in \Sigma^\prime} I_\sigma$ since if $\sigma\subseteq \tau'$ for some $\sigma\in \Sigma'$ then $\sigma\subseteq \tau$ contradicting that $\Sigma$ is an antichain. Whitehead's lemma \cite[Lemma  7.4]{WhiteheadThesis} then shows  that $J\not\subseteq Q$.
	A similar antichain argument shows that $I_\rho$ is an irredundant summand in $Q$. From \cite[Corollary 5.3]{deconcini1980young} it follows that the ideal $Q$ is  $D_p$--primary.  Meanwhile any primary component of $J := I_{\tau^\prime} +\sum_{\sigma\in \Sigma^\prime}I_\sigma $ will be $D_t$--primary for some  $t\geqslant p$ by the inductive hypothesis. This leads to the desired conclusion that $D_p$ is the maximal associated prime of $I_\Sigma$.	\end{proof}

\subsection{Symbolic powers of $G$-invariant ideals}

 In this section we give a closed formula for the symbolic powers of a $G$-invariant ideal. It  relies on the notion of Yamanouchi filtrations introduced in \Cref{def: Yamanouchi filtration} and  the truncation operation  $\l_{\geq p}$ from \Cref{def: <v}, which does not seem to have appeared previously in the literature.

\begin{definition}\label{def: Yamanouchi filtration}
For a set $\Sigma\subset H_r$, write $\Sigma^\vee:=\{\sigma^\vee:\sigma\in\Sigma\}$. We say that $\l\in H_r$ {\em admits a Yamanouchi $m$-filtration} with content $\Sigma^\vee$ whenever there is a chain
\[
\emptyset=\lambda^{(0)}\subseteq\lambda^{(1)}\subseteq\cdots\subseteq\lambda^{(m)}=\lambda
\]
such that, for each $0\leq i\leq m-1$, the skew shape $\lambda^{(i+1)}/\lambda^{(i)}$ admits a Yamanouchi filling with content ${\sigma^{(i+1)}}^\vee$ for some $\sigma^{(i+1)}\in\Sigma$.

We denote the set of diagrams that admit a Yamanouchi $m$-filtration with content $\Sigma^\vee$ by $\Lambda_m(\Sigma)$. When $\Sigma=\{\s\}$ is a singleton we denote this set by $\Lambda_m(\sigma)$.
\end{definition}

\begin{example}\label{ex: Yamanouchi filtration}
Let $\s=(3,1)\in H_5$. We show below that $\lambda = (5,4,2,1)\in \Lambda_3(\s)$.
\begin{figure}[h]
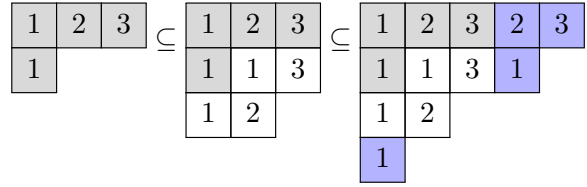

\centering
\[
\begin{ytableau}
*(gray!30) 1 & *(gray!30) 2 & *(gray!30) 3 \\
*(gray!30) 1
\end{ytableau}
\subseteq
\begin{ytableau}
*(gray!30) 1 & *(gray!30) 2 & *(gray!30) 3  \\
*(gray!30) 1 & 1 &3 \\
1 &2
\end{ytableau}
\subseteq
\begin{ytableau}
*(gray!30) 1 & *(gray!30) 2 & *(gray!30) 3  & *(blue!30) 2 & *(blue!30) 3 \\
*(gray!30) 1 & 1  &3 &*(blue!30) 1 \\
1 &2 \\
*(blue!30)  1
\end{ytableau}
\]
\caption{A Yamanouchi $3$-filtration demonstrating $\lambda = (5,4,2,1)\in \Lambda_3(3,1)$. }
\end{figure}
\end{example}

In the following we analyze the notion of Yamanouchi $m$-filtration in more detail, showing that the sets $\Lambda_m(\Sigma)$ of diagrams admitting Yamanouchi $m$-filtrations of content $\Sigma^\vee$ can be characterized in terms of Littlewood--Richardson coefficients.

\begin{lemma}\label{lemma Lambda=littlewood rich}
With notation as in \Cref{def: Yamanouchi filtration},  the following are equivalent for $m\geq 2$:
\begin{enumerate}
\item $\l\in \Lambda_m(\Sigma)$,
\item $\l\in H_r$ admits a filtration
$
\emptyset=\l^{(0)}\subseteq \lambda^{(1)}\subseteq\cdots\subseteq\lambda^{(m)}=\l
$
which satisfies, for some not necessarily distinct $\s^{(1)},\s^{(2)},\ldots,\s^{(m)}\in\Sigma$, the condition $\lambda^{(1)}=\sigma^{(1)}$ and
\[
c^{{\l^{(i+1)}}^\vee}_{{\l^{(i)}}^\vee,{\s^{(i+1)}}^\vee}\neq 0 \quad \text{for } 1\leq i\leq m-1,
\]
\item there exist   (not necessarily distinct) $\s^{(1)}, \s^{(2)}, \ldots, \s^{(m)}\in \Sigma$ such that
\[
c_{{\s^{(1)}}^\vee,{\s^{(2)}}^\vee,\cdots, {\s^{(m)}}^\vee}^{\l^\vee}\neq 0.
\]
\end{enumerate}
\end{lemma}

\begin{proof}
$(1) \Leftrightarrow (2)$
Consider the filtration of $\l$ given in \Cref{def: Yamanouchi filtration}.
Fix $i$ in the given range and let $\mu=\l^{(i+1)}, \nu=\l^{(i)}$.
 Dualization sends the skew shape $\mu/\nu$ to $\mu^\vee/\nu^\vee$.
As noted at the end of  \cref{s: Schur}, under dualization, column-reading words become row-reading words, row-increasing tableaux become column-increasing tableaux, and the Yamanouchi condition becomes the Littlewood--Richardson condition, while tableau fillings do not change.

By  \cite[Proposition 3, p.~64]{FULTON_TABLEAUX}, $c^{\mu^\vee}_{\nu^\vee,\s^\vee}\neq 0$ if and only if there exists a Littlewood--Richardson tableau of shape $\mu^\vee/\nu^\vee$ and content $\sigma^\vee$. This holds if and only if $\mu/\nu$ admits a Yamanouchi filling with content $\s^\vee$, which is precisely the condition needed for the given filtration to correspond to an element of $\Lambda_m(\Sigma)$.

$(2) \Leftrightarrow (3)$
We argue by induction on $m$. If $m=2$, the condition $\lambda^{(1)}=\sigma^{(1)}$, together with $c^{\lambda^\vee}_{{\lambda^{(1)}}^\vee,{\sigma^{(2)}}^\vee}\neq0$, is exactly the condition $c^{\lambda^\vee}_{{\sigma^{(1)}}^\vee,{\sigma^{(2)}}^\vee}\neq0$.

Assume $m>2$. If (2) holds, the induction hypothesis applied to the first $m-1$ stages gives
$c^{{\lambda^{(m-1)}}^\vee}_{{\sigma^{(1)}}^\vee,\cdots,{\sigma^{(m-1)}}^\vee}\neq0$.
Together with $c^{\lambda^\vee}_{{\lambda^{(m-1)}}^\vee,{\sigma^{(m)}}^\vee}\neq0$, this shows that $s_{\lambda^\vee}$ occurs with nonzero coefficient in $s_{{\sigma^{(1)}}^\vee}\cdots s_{{\sigma^{(m)}}^\vee}$, and hence (3) holds.

Conversely, suppose that $c^{\lambda^\vee}_{{\sigma^{(1)}}^\vee,\cdots,{\sigma^{(m)}}^\vee}\neq0$. By associativity of multiplication of Schur functions,
\[
c^{\lambda^\vee}_{{\sigma^{(1)}}^\vee,\cdots,{\sigma^{(m)}}^\vee}
=\sum_{\nu}c^{\nu^\vee}_{{\sigma^{(1)}}^\vee,\cdots,{\sigma^{(m-1)}}^\vee}c^{\lambda^\vee}_{\nu^\vee,{\sigma^{(m)}}^\vee}.
\]
All Littlewood--Richardson coefficients are nonnegative, so there is a diagram $\nu$ for which both factors are nonzero. By induction, the first factor supplies a filtration from $\emptyset$ to $\nu$ satisfying (2). Appending $\lambda^{(m)}=\lambda$ and using $c^{\lambda^\vee}_{\nu^\vee,{\sigma^{(m)}}^\vee}\neq0$ completes the required filtration.
\end{proof}

\begin{remark}
\Cref{lemma Lambda=littlewood rich} gives the equivalent description
\begin{equation}\label{eq: Lambda}
\Lambda_m(\Sigma)= \left\{ \l \in H_r : c_{{\s^{(1)}}^\vee,{\s^{(2)}}^\vee,\cdots, {\s^{(m)}}^\vee}^{\l^\vee}\neq 0 \text{ for some } \s^{(1)}, \s^{(2)}, \ldots, \s^{(m)}\in \Sigma \right \}.
\end{equation}
We also denote
\[
\Lambda_m(\Sigma)_{\geq t}=\{\l_{\geq t} : \l\in \Lambda_m(\Sigma)\}.
\]
\end{remark}

\noindent We now arrive at our first main result: a description for the symbolic powers of $G$-invariant ideals.

\begin{theorem}\label{lemma: Symbolic power one diagram invariant}
Let $\Sigma\subset H_r$ be a set of Young diagrams and suppose $1\leq t\leq r$.
With notation as in \eqref{eq: Lambda}, for every positive integer $m$ we have:
\begin{enumerate}
\item $I_\Sigma^m=\displaystyle \sum_{\lambda\in\Lambda_m(\Sigma)} I_\lambda$;
\item $I_\Sigma^{(m)_t}= \displaystyle \sum_{\lambda\in\Lambda_m(\Sigma)} I_{\lambda_{\geq t}}$,
\item $\gamma_j \left(I_\Sigma^{(m)_t}\right)= \min \left \{\gamma_j(\l_{\geq t}) : \lambda\in\Lambda_m(\Sigma) \right \} $ for $1\leq j\leq r$.
\end{enumerate}
\end{theorem}

\begin{proof}

(1) The claim follows by induction on $m$ using Theorem~\ref{theorem: product whitehead}.

Indeed the condition that
$
\emptyset=\lambda^{(0)}\subseteq \lambda^{(1)}\subseteq\cdots\subseteq\lambda^{(m)}=\lambda
$
is a Yamanouchi $m$-filtration with content $\Sigma^\vee$ says that $\lambda^{(1)}\in \Sigma$, which settles the case $m=1$. When $m>1$ it says that $\lambda^{(m)}\in D(\lambda^{(m-1)},\sigma^{(m)})$  for some $\sigma^{(m)} \in \Sigma$ according to \Cref{def: D}. Since $I_{\lambda^{(m-1)}}\subseteq I_\Sigma^{m-1}$ by the inductive hypothesis, \Cref{theorem: product whitehead} gives
\[
I_\lambda\subseteq I_{\lambda^{(m-1)}}I_{\sigma^{(m)}}
\subseteq I_\Sigma^{m-1}I_{\sigma^{(m)}}
\subseteq I_\Sigma^m.
\]

Conversely, the inductive hypothesis gives
\[
I_\Sigma^{m-1}=\sum_{\mu\in\Lambda_{m-1}(\Sigma)}I_\mu,
\]
and hence
\[
I_\Sigma^m=I_\Sigma^{m-1}I_\Sigma
 =\sum_{\substack{\mu\in\Lambda_{m-1}(\Sigma)\\ \sigma\in\Sigma}}I_\mu I_\sigma.
\]
By \Cref{theorem: product whitehead}, every isotypic ideal $I_\lambda$ occurring in this sum satisfies $\lambda\in D(\mu,\sigma)$ for some $\mu\in\Lambda_{m-1}(\Sigma)$ and $\sigma\in\Sigma$. Concatenating a Yamanouchi $(m-1)$-filtration of $\mu$ with a Yamanouchi filling of $\lambda/\mu$ of content $\sigma^\vee$ produces a Yamanouchi $m$-filtration of $\lambda$. Therefore $\lambda\in\Lambda_m(\Sigma)$.

(2) We first dispose of the case in which $I_\Sigma\not\subseteq D_t$ and thus $
I_\Sigma^{(m)_t}=I_\Sigma^mR_{D_t}\cap R=R$. Since $I_\Sigma^m\not\subseteq D_t$, there is some
$\lambda\in\Lambda_m(\Sigma)$ such that $I_\lambda\not\subseteq D_t$.
Equivalently, $\lambda_1<t$, and therefore
$\lambda_{\geq t}=\varnothing$. With the convention
$I_\varnothing:=R$,
the sum $\sum_{\lambda\in\Lambda_m(\Sigma)}I_{\lambda_{\geq t}}$
 equals $R$.

Now assume that $I_\Sigma \subseteq D_t$.
Set
\[
p(t)=\min\{\lambda_i:\lambda\in\Lambda_m(\Sigma),\ \lambda_i\geq t\},
\]
the shortest row length at least $t$ occurring among the diagrams in $\Lambda_m(\Sigma)$. Then \Cref{def: <v} gives
$I_{\lambda_{\geq t}}=I_{\lambda_{\geq p(t)}}$ for every $\lambda\in\Lambda_m(\Sigma)$.

Applying Lemma~\ref{lem: <v} and Lemma~\ref{lemma: sum and intersection}, we write
\[
I_\Sigma^m =  \sum_{\lambda \in \Lambda_m(\Sigma)} I_\lambda
=\sum_{\lambda\in\Lambda_m(\Sigma)}(I_{\lambda_{\ge p(t)}}\cap I_{\lambda_{<  p(t)}})
=\Bigl(\sum_{\lambda\in\Lambda_m(\Sigma)} I_{\lambda_{\ge p(t)}}\Bigr)\cap\Bigl(\bigcap_i J_i\Bigr),
\]
where each $J_i$ contains at least one $I_{\lambda_{< p(t)}}$  as a summand for some $\lambda \in \Lambda_m(\Sigma)$.  By definition of $p(t)$ we see that the row length of $\lambda_{< p(t)}$ is strictly smaller than $t$. Thus $I_{\lambda_{< p(t)}}\not\subseteq D_t$ and therefore $J_i\not\subseteq D_t$ for each $i$. By contrast,  by \Cref{lem: <v} and \Cref{lemma: symbolic power for sum of ideal invariant}, the first ideal in the decomposition above has $D_{p(t)}$ as its unique maximal associated prime and it satisfies $D_{p(t)}\subseteq D_t$.
The generalized symbolic power of interest is according to \Cref{def: symbolic_power}
 $$I_\Sigma^{(m)_t}= \sum_{\lambda \in \Lambda_m(\Sigma)}  I_{\lambda_{\geq p(t)}} = \sum_{\lambda \in \Lambda_m(\Sigma)}  I_{\lambda_{\geq t}} .$$

(3) From \eqref{eq: nu}, part~(2), and properties of valuations we deduce
\[
\gamma_j\bigl(I_\Sigma^{(m)_t}\bigr)
 =\min\{\gamma_j(I_{\lambda_{\geq t}}):\lambda\in\Lambda_m(\Sigma)\}.
\]
It remains to show that $\gamma_j(I_{\lambda_{\geq t}})=\gamma_j(\lambda_{\geq t})$, where $\gamma_j$ takes the valuative meaning  \eqref{eq: nu} on the left side and the combinatorial meaning \eqref{eq: gamma} on the right. The standard monomial $K_{\lambda_{\geq t}}$ belongs to $I_{\lambda_{\geq t}}$ by \Cref{def: I_sigma} and satisfies
$\gamma_j(K_{\lambda_{\geq t}})=\gamma_j(\lambda_{\geq t})$; hence
$\gamma_j(I_{\lambda_{\geq t}})\leq\gamma_j(\lambda_{\geq t})$.
On the other hand,
\[
I_{\lambda_{\geq t}}\subseteq\overline{I_{\lambda_{\geq t}}}=D_{\lambda_{\geq t}},
\]
where the equality follows from \cite[Theorem 8.1]{deconcini1980young}. Therefore
$\gamma_j(I_{\lambda_{\geq t}})\geq\gamma_j(D_{\lambda_{\geq t}})=\gamma_j(\lambda_{\geq t})$.
The two inequalities complete the proof.
	\end{proof}

	\begin{example}\label{ex: symb 1}
	For $r=5$ and $\s=(3,2)$  we first compute $I_\s^3$; this will be used to obtain the third generalized symbolic power of $I_\s$. Towards this end a computation with the package \texttt{SchurRings} in \cite{M2} gives
	\begin{multline*}
	s_{(3,2)^\vee}^3= 2\,s_{(5,5,5)^\vee}+
			12\,s_{(5,5,4,1)^\vee}+
			21\,s_{(5,5,3,2)^\vee}+
			15\,s_{(5,5,3,1,1)^\vee}+
			14\,s_{(5,5,2,2,1)^\vee}+
			4\,s_{(5,5,2,1,1,1)^\vee}+\\
			15\,s_{(5,4,4,2)^\vee}+
			11\,s_{(5,4,4,1,1)^\vee}+
      			14\,s_{(5,4,3,3)^\vee} +
      			32\,s_{(5,4,3,2,1)^\vee}+
			8\,s_{(5,4,3,1,1,1)^\vee}+
			11\,s_{(5,4,2,2,2)^\vee}+\\
			9\,s_{(5,4,2,2,1,1)^\vee}+
			11\,s_{(5,3,3,3,1)^\vee}+
			12\,s_{(5,3,3,2,2)^\vee}+
			9\,s_{(5,3,3,2,1,1)^\vee}+
      			6\,s_{(5,3,2,2,2,1)^\vee}+
     			s_{(5,2,2,2,2,2)^\vee}+\\
			4\,s_{(4,4,4,3)^\vee}+
			8\,s_{(4,4,4,2,1)^\vee}+
			2\,s_{(4,4,4,1,1,1)^\vee}+
			9\,s_{(4,4,3,3,1)^\vee}+
			9\,s_{(4,4,3,2,2)^\vee}+
			8\,s_{(4,4,3,2,1,1)^\vee}+\\
			4\,s_{(4,4,2,2,2,1)^\vee}+
			6\,s_{(4,3,3,3,2)^\vee}+
			4\,s_{(4,3,3,3,1,1)^\vee}+
			6\,s_{(4,3,3,2,2,1)^\vee}+
			2\,s_{(4,3,2,2,2,2)^\vee}+
			s_{(3,3,3,3,3)^\vee}+\\
			2\,s_{(3,3,3,3,2,1)^\vee}+
			s_{(3,3,3,2,2,2)^\vee}.
       \end{multline*}
      Therefore \eqref{eq: Lambda}  and  \Cref{lemma: Symbolic power one diagram invariant}~(1) give 
    \begin{multline*}
    I_{(3,2)}^3= I_{(5,5,5)}+
			 I_{(5,5,4,1)}+
			 I_{(5,5,3,2)}+
			 I_{(5,5,3,1,1)}+
			 I_{(5,5,2,2,1)}+
			 I_{(5,5,2,1,1,1)}+\\
			 I_{(5,4,4,2)}+
			 I_{(5,4,4,1,1)}+
      			 I_{(5,4,3,3)} +
      			 I_{(5,4,3,2,1)}+
			 I_{(5,4,3,1,1,1)}+
			 I_{(5,4,2,2,2)}+\\
			 I_{(5,4,2,2,1,1)}+
			 I_{(5,3,3,3,1)}+
			 I_{(5,3,3,2,2)}+
			 I_{(5,3,3,2,1,1)}+
      			 I_{(5,3,2,2,2,1)}+
     			 I_{(5,2,2,2,2,2)}+\\
			 I_{(4,4,4,3)}+
			 I_{(4,4,4,2,1)}+
			 I_{(4,4,4,1,1,1)}+
			 I_{(4,4,3,3,1)}+
			 I_{(4,4,3,2,2)}+
			 I_{(4,4,3,2,1,1)}+\\
			 I_{(4,4,2,2,2,1)}+
			 I_{(4,3,3,3,2)}+
			 I_{(4,3,3,3,1,1)}+
			 I_{(4,3,3,2,2,1)}+
			 I_{(4,3,2,2,2,2)}+
			 I_{(3,3,3,3,3)}+
			 I_{(3,3,3,3,2,1)}+
			 I_{(3,3,3,2,2,2)}.
\end{multline*}
 Applying \Cref{lemma: Symbolic power one diagram invariant}~(2) gives an expression involving  32 terms  for $I_{(3,2)}^{(3)_2}$, corresponding to removing the parts of size 1 from the partitions listed in the sum for  $I_{(3,2)}^3$ above. Rather than list these terms, we will revisit the computation of the symbolic powers in \Cref{ex: 2} with an alternate formula that yields only the 14 irredundant terms necessary to describe $I_{(3,2)}^{(3)_2}$.
\end{example}

To compute the asymptotic invariants of isotypic ideals  $I_\s$ it will be convenient to relate them to the asymptotic invariants of $I_{\s_{\leftarrow p}}$ for a smaller partition $\s_{\leftarrow p}$. This will also lead to a formula for generalized symbolic powers of $G$-invariant ideals that is more computationally efficient.

\begin{notation}\label{not: s'}
For a Young diagram $\sigma$ and a positive integer $p$, define
\[
\sigma_{\leftarrow p}
=\bigl(\sigma_i-p+1:\sigma_i\geq p\bigr).
\]
Thus $\sigma_{\leftarrow p}$ is obtained by discarding rows shorter than $p$ then deleting the first $p-1$ columns; it may be the empty diagram. Conversely, for a Young diagram $\mu$, define
\[
\mu_{\rightarrow p}=(\mu_i+p-1)_i.
\]
Then $(\mu_{\rightarrow p})_{\leftarrow p}=\mu$ for every $\mu$, while
$(\sigma_{\leftarrow p})_{\rightarrow p}=\sigma$ precisely when every row of $\sigma$ has length at least $p$ (with the same statement for the empty diagram under the evident convention).

For a set $\Sigma$ of Young diagrams, set
\[
\Sigma_{\leftarrow p}=\{\sigma_{\leftarrow p}:\sigma\in\Sigma\}
\quad \text{ and } \quad
\Sigma_{\rightarrow p}=\{\sigma_{\rightarrow p}:\sigma\in\Sigma\}.
\]

\Cref{fig5} illustrates this correspondence.

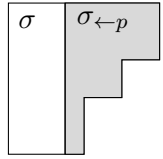
\begin{figure}[h]
\begin{tikzpicture}[scale=0.5]
			\draw[black] (0,0) -- (4,0) -- ( 4,-1.5)  -- (3,-1.5) -- (3,-2.5) -- (2,-2.5) --(2,-4)-- (0,-4)  -- cycle;
			\draw[fill=gray!30] (1.5,0) -- (4,0) -- ( 4,-1.5)  -- (3,-1.5) -- (3,-2.5) -- (2,-2.5) --(2,-4)-- (1.5,-4)  -- cycle;
			\node at  (0.5,-.5){ $\sigma$ };
			\draw [black]  (1.5,0) -- (1.5,-4);
			\node at  (2.5,-.5){ $\sigma_{\leftarrow p}$ };
\end{tikzpicture}
		\caption{Partition $\s_{\leftarrow p}$ (gray) associated to the partition $\s$ (white and gray).}\label{fig5}
\end{figure}

\end{notation}

In the next result we are able to relate the symbolic powers of $\Sigma$ and $\Sigma_{\leftarrow p}$ as well as the sets of diagrams that have Yamanouchi $m$-filtrations with content $\Sigma^\vee$ and $\Sigma_{\leftarrow p}^\vee$, respectively. This leads to a more economical computation of the generalized symbolic powers as described in \Cref{prop: sigma'}~(3) below.  In the formula in part (3) each summand is irredundant, assuming that $\l$ ranges over the minimal partitions  in $\Lambda_m(\Sigma_{\leftarrow t})$.

\begin{theorem}\label{prop: sigma'}
Let $\Sigma\subset H_r$ be a set of Young diagrams and let $1\leq p\leq t\leq r$ be positive integers. Then, using \Cref{def: Yamanouchi filtration} and \Cref{not: s'}, the following hold for each integer $m\geq1$:
\begin{enumerate}
\item the map $\l\mapsto\l_{\leftarrow p}$ gives a bijection between the minimal diagrams, with respect to containment, in
\[
\Lambda_m(\Sigma)_{\geq t}:=\{\lambda_{\geq t}:\lambda\in\Lambda_m(\Sigma) \}
\quad \text{ and } \quad
\Lambda_m(\Sigma_{\leftarrow p})_{\geq t-p+1}
 :=\{\nu_{\geq t-p+1}:\nu\in\Lambda_m(\Sigma_{\leftarrow p})\},
\]
\item  there are equalities
\[
\gamma_j\bigl(I_\Sigma^{(m)_t}\bigr)
 =\gamma_{j-p+1}\bigl(I_{\Sigma_{\leftarrow p}}^{(m)_{t-p+1}}\bigr)
 \qquad\text{for all }j \text{ such that } p\leq j\leq r,
\]
where the former ideal is computed with respect to a $r\times n$ matrix and the latter with respect to a $(r-p+1)\times (n-p+1)$ matrix,
\item an alternate formula for the generalized symbolic powers is
\[
I_\Sigma^{(m)_t}=\sum_{\lambda\in\Lambda_m(\Sigma_{\leftarrow t})}I_{\lambda_{\rightarrow t}}.
\]
\end{enumerate}
\end{theorem}

\begin{proof}
(1) We proceed by induction on $p$. The case $p=1$ is tautological, because $\Sigma_{\leftarrow 1}=\Sigma$.

Assume $p\geq2$.
Let $X=(x_{ij})$ be an $r\times n$ matrix of indeterminates, and let
$
Y=(y_{ij})_{\substack{2\leq i\leq r\\2\leq j\leq n}}
$
be an $(r-1)\times(n-1)$ matrix of indeterminates over $\kk$. Define a ring homomorphis m$\Phi$ by
\begin{align*}
x_{ij}&\longmapsto y_{ij}+x_{1j}x_{i1}x_{11}^{-1}
&& (2\leq i\leq r,\ 2\leq j\leq n),\\
x_{1j}&\longmapsto x_{1j},\qquad
x_{i1}\longmapsto x_{i1}
&& (1\leq j\leq n,\ 1\leq i\leq r).
\end{align*}
This gives as in \cite[Lemma 3.4.5]{bruns2022determinants}  an isomorphism
\begin{equation}\label{eq: iso}
S:=\kk[x_{ij}][x_{11}^{-1}]
\overset{\Phi}{\cong}
\kk[y_{ij}:2\leq i\leq r,\ 2\leq j\leq n]
[x_{11},\ldots,x_{1n},x_{21},\ldots,x_{r1}][x_{11}^{-1}]=:S'.
\end{equation}

 Elementary row and column operations with pivot $x_{11}$ transform $\Phi(X)$ into
\[
X'=\begin{pmatrix}x_{11}&0\\0&Y\end{pmatrix}.
\]
Thus $\Phi$ identifies the extended determinantal ideals $D_t(X)S$ and $D_{t-1}(Y)S'$ via $\Phi(D_t(X)S)=x_{11}D_{t-1}(Y)S'=D_{t-1}(Y)S'$.
Moreover, if $s$ is the number of rows of $\sigma$, then
\[
\Phi(K_\sigma(X))=x_{11}^sK_{\sigma_{\leftarrow2}}(Y),
\]
and consequently, since  the action of $\GL(X')$  fixes $x_{11}$ and restricts to an action on $Y$, we have
\begin{equation}\label{eq: iso 2}
\Phi(I_\sigma(X)S)
=x_{11}^sI_{\sigma_{\leftarrow2}}(Y)S'
= I_{\sigma_{\leftarrow2}}(Y)S'.
\end{equation}
Summing over $\sigma\in\Sigma$ gives
\begin{equation}\label{eq: iso 2 Sigma}
\Phi(I_\Sigma(X)S)= I_{\Sigma_{\leftarrow2}}(Y)S'.
\end{equation}

Repeating the one-pivot step described above $p-1$ times produces an $(r-p+1)\times(n-p+1)$ matrix $W=[w_{ij}]$ and an isomorphism $\Psi: R'\to S''$ where
\begin{eqnarray*}
R' &=&\kk[x_{ij}: 1\leq i\leq r, 1\leq j\leq n][\Delta_1^{-1},\ldots,\Delta_{p-1}^{-1}], \\
A &=& \kk[w_{ij} : p\leq i\leq r, p\leq j\leq n]\\
S'' &=& A[z_1, \ldots, z_N] [u_1^{\pm1}, \ldots, u_{p-1}^{\pm1}],
\end{eqnarray*}
the elements $\Delta_i$ are the leading principal minors of $X$ with the convention $\Delta_0=1$,  $u_i=\Delta_i/\Delta_{i-1}$ are the successive pivot entries and $z_i$ are the successive border entries. These elements $z_1, \ldots, z_N, u_1, \ldots, u_{p-1}$ are algebraically independent over $A$ thus $S''$ is a free extension of $A$ with basis given by the monomials in these variables (Laurent monomials in the $u_i$'s) .

Induction on $p$ with base case given by \eqref{eq: iso 2 Sigma} for $p=2$ yields the analogous identity
\begin{equation}\label{eq: iso 2' Sigma}
\Psi(I_\Sigma(X)R')= I_{\Sigma_{\leftarrow p}}(W)S''.
\end{equation}

Since $\Psi(D_t(X)R')=D_{t-p+1}(W)S''$, taking powers in \eqref{eq: iso 2' Sigma} and localizing gives
\begin{equation}\label{eq: iso 3 Sigma}
\Psi(I_\Sigma(X)^{(m)_t}R') = \Psi(I_\Sigma(X)^m R'_{D_t(X)R'} \cap R')
=  I_{\Sigma_{\leftarrow p}}(W)^m S''_{D_{t-p+1}(W)S''} \cap S'' = I_{\Sigma_{\leftarrow p}}(W)^{(m)_{t-p+1}}S''.
\end{equation}
The first equality in \eqref{eq: iso 3 Sigma} can be justified as follows.
Since \(p\leq t\), each \(\Delta_i\), for \(i\leq p-1\), has size
strictly smaller than \(t\), and hence \(\Delta_i\notin D_t(X)\).
Consequently
$
R_{D_t(X)}= R'_{D_t(X)R'}
$
and if \(L\subseteq R\) is an ideal then
\[
\bigl(LR_{D_t(X)}\cap R\bigr)R'
=
LR_{D_t(X)}\cap R'
=
LR'_{D_t(X)R'}\cap R'.
\]
Indeed, only the first equality requires explanation. If
\(f/u\in LR_{D_t(X)}\cap R'\), with \(f\in R\) and \(u\) is in the multiplicative set generated by $\Delta_1, \ldots, \Delta_{p-1}$, then
\(u\) is a unit in \(R_{D_t(X)}\), so \(f\in LR_{D_t(X)}\cap R\).
Thus \(f/u\) belongs to the left-hand side; the reverse containment stems from the last equality.

The last equality in \eqref{eq: iso 3 Sigma} holds because $S''$ is a free, hence faithfully flat extension of $A$ and every associated prime of a $\GL$--invariant ideal of $A$ extends to a prime ideal of $S''$ \cite[Proposition 2.1]{Walker}.

Using \Cref{lemma: Symbolic power one diagram invariant}~(2)   we rewrite \eqref{eq: iso 3 Sigma}  as
\[
\sum_{\s\in \Lambda_m(\Sigma)_{\geq t}} \Psi(I_\s R')
=
\sum_{\s\in \Lambda_m(\Sigma_{\leftarrow p})_{\geq {t-p+1}}}I_\s(W) S''.
\]
Using the identity \eqref{eq: iso 2' Sigma}, we conclude
\begin{equation}\label{eq: iso 4}
\sum_{\s\in (\Lambda_m(\Sigma)_{\geq t})_{\leftarrow p}} I_\s(W) S''
=
\sum_{\s\in \Lambda_m(\Sigma_{\leftarrow p})_{\geq {t-p+1}}}I_\s(W) S''.
\end{equation}
Since the ring $S''$ is a free, hence a faithfully flat extension of $A$, we conclude by contraction to $A$ the corresponding equality in $A$, that is,
 \begin{equation}\label{eq: iso 5}
\sum_{\s\in (\Lambda_m(\Sigma)_{\geq t})_{\leftarrow p}} I_\s(W)
=
\sum_{\s\in \Lambda_m(\Sigma_{\leftarrow p})_{\geq {t-p+1}}}I_\s(W).
\end{equation}
By uniqueness of minimal isotypic decompositions, the minimal elements of the indexing sets in \eqref{eq: iso 5} must be equal.

(2) By \Cref{lemma: Symbolic power one diagram invariant}~(2) we have
\begin{eqnarray*}
 I_\Sigma(X)^{(m)_t} &=& \sum_{\s\in \Lambda_m(\Sigma)_{\geq t}} I_\s\\
  I_{\Sigma_{\leftarrow p}}(W)^{(m)_{t-p+1}}&=& \sum_{\s\in \Lambda_m(\Sigma_{\leftarrow p})_{\geq {t-p+1}}}I_\s.
\end{eqnarray*}
The formulas for the valuations now follow from part (1) using that $\gamma_j(\mu_{\rightarrow p})=\gamma_{j-p+1}(\mu)$.

(3)
The formula arises as an application of \Cref{lemma: Symbolic power one diagram invariant}~(2) and part (1). Indeed part (1) applied with  $p=t$ yields that the minimal elements of the sets $\Lambda_m(\Sigma)_{\geq t}$ and $(\Lambda_m(\Sigma_{\leftarrow t}))_{\rightarrow t}$ are the same. Passing to the sum of isotypic ideals indexed by these partitions and applying the formula for symbolic powers in \Cref{lemma: Symbolic power one diagram invariant}~(2) yields the claim.
\end{proof}

The following computation revisits \Cref{ex: symb 1}, applying the  formula in  \Cref{prop: sigma'}~(3).

\begin{example}\label{ex: 2}

For $r=5$ and $\s=(3,2)$  we will determine the third generalized symbolic power of $I_\s$ using  \Cref{prop: sigma'}~(3)  with
\[
\sigma_{\leftarrow1}=\sigma,\qquad
\sigma_{\leftarrow2}=(2,1),\qquad
\sigma_{\leftarrow3}=(1),\qquad
\sigma_{\leftarrow t}=\varnothing\ \text{ for }t\geq4.
\]
The relevant Schur-function computations are
\begin{align*}
s_{(\sigma_{\leftarrow2})^\vee}^3=s_{(2,1)^\vee}^3
 &=4s_{(4,4,1)^\vee}+8s_{(4,3,2)^\vee}+9s_{(4,3,1,1)^\vee}
   +9s_{(4,2,2,1)^\vee}+6s_{(4,2,1,1,1)^\vee}\\
 &\quad+s_{(4,1,1,1,1,1)^\vee}+2s_{(3,3,3)^\vee}+8s_{(3,3,2,1)^\vee}
   +4s_{(3,3,1,1,1)^\vee}\\
 &\quad+4s_{(3,2,2,2)^\vee}+6s_{(3,2,2,1,1)^\vee}
   +2s_{(3,2,1,1,1,1)^\vee}\\
 &\quad+2s_{(2,2,2,2,1)^\vee}+s_{(2,2,2,1,1,1)^\vee},\\
s_{(\sigma_{\leftarrow3})^\vee}^3=s_{(1)^\vee}^3
 &=s_{(3)^\vee}+2s_{(2,1)^\vee}+s_{(1,1,1)^\vee}.
\end{align*}
Applying the shift operations ${-}_{\rightarrow2}$ and ${-}_{\rightarrow3}$ to
the minimal diagrams in these two expansions yields
\begin{align*}
\bigl(I_{(3,2)}\bigr)_{\Ass}^{(3)}=I_{(3,2)}^{(3)_2}
 &=I_{(5,5,2)}+I_{(5,4,3)}+I_{(5,4,2,2)}+I_{(5,3,3,2)}
   +I_{(5,3,2,2,2)}\\
 &\quad+I_{(5,2,2,2,2,2)}+I_{(4,4,4)}+I_{(4,4,3,2)}
   +I_{(4,4,2,2,2)}\\
 &\quad+I_{(4,3,3,3)}+I_{(4,3,3,2,2)}+I_{(4,3,2,2,2,2)}
   +I_{(3,3,3,3,2)}+I_{(3,3,3,2,2,2)},\\
\bigl(I_{(3,2)}\bigr)_{\Min}^{(3)}=I_{(3,2)}^{(3)_3}
 &=I_{(5)}+I_{(4,3)}+I_{(3,3,3)}.
\end{align*}
For $t=1$ the ideal $I_{(3,2)}^{(3)_1}=I_{(3,2)}^{3}$ was determined in \Cref{ex: symb 1} and for $t\geq4$ the ideals $I_{(3,2)}^{(3)_t}=R$ because $I_{(3,2)}\not\subseteq D_4$, equivalently  $\sigma_{\leftarrow t}=\varnothing$ for $t\geq4$.

In each of the sums above the partitions listed form antichains, so the corresponding ideals are irredundant summands. Moreover, the smaller sizes of the partitions $\sigma_{\leftarrow t}$ present a computational advantage over  \Cref{ex: symb 1}.
\end{example}

\section{Asymptotic Invariants of GL-Invariant Ideals}\label{s: asymptotic}
We are now ready to determine the asymptotic invariants of the symbolic powers of  $G$-invariant ideals. We will  first determine the skew Waldschmidt constants corresponding to the valuations $\gamma_j$ and next use \Cref{thm: DFMS} to compute the asymptotic resurgence.

The following observation is implicit in \cite[Corollary 3.16]{bdhmrational}. The similarity between
\Cref{thm: Rees_valuations I_Sigma} and \Cref{thm: Rees_valuations} reflects the fact that taking integral closure preserves the Rees valuations and the relationship $\overline{I_\Sigma}=\overline{D_\Sigma}$.

\begin{lemma}\label{thm: Rees_valuations I_Sigma}
Let $\Sigma\subset H_r$ be finite and nonempty, and let $I_\Sigma=\sum_{\sigma\in\Sigma}I_\sigma$. The  facets not contained in a coordinate hyperplane of
$\P_\Sigma=P_\Sigma+\RR_{\geq0}^r$  correspond to the Rees valuations of $I_\Sigma$. More precisely, if such a facet has supporting equation
$
\sum_{i=1}^r h_ix_i=c,
$
written with  $h=(h_1,\ldots,h_r)\in\ZZ_{\geq0}^r$ and $c\in\ZZ_{>0}$, then the corresponding Rees valuation is $\sum_{i=1}^rh_i\gamma_i$.
\end{lemma}

\begin{proof}
In \cite[Theorem 8.2]{deconcini1980young} (see also \eqref{eq: integral closure}) formulas for the integral closures of sums of  isotypic ideals are given, including  $\overline{I_\Sigma}=\overline{D_\Sigma}$. Consequently it follows that $\overline{I_\Sigma^m}=\overline{D_\Sigma^m}$ for all $m\geq 1$. As the Rees valuations only depend on the integral closure of an ideal and its powers, this shows that the Rees valuations of $I_\Sigma$ and $D_\Sigma$ coincide. We conclude using  \Cref{thm: Rees_valuations}.
\end{proof}

We can now determine the Waldschmidt constant for  $G$-invariant ideals.

\begin{theorem}\label{theorem: waldschmidt constant one diagram}
Let $\Sigma\subset H_r$ be a set of Young diagrams. Then for any positive integer $p\leq r$ the skew
Waldschmidt constants for the family of generalized symbolic powers $I_\Sigma^{(\bullet)_p}$, denoted  $\widehat{\gamma_j}^p(I_\Sigma)$ are given by
\[
\widehat{\gamma_j}^p(I_\Sigma):=\lim_{m\to\infty}\frac{\gamma_j\left (I_\Sigma^{(m)_p}\right )}{m}=
\begin{cases}
\frac{(r-j+1) \ \gamma_{p}(\Sigma)}{r-p+1} & \text{ for } 1\leq j\leq p,\\
\gamma_j(\Sigma)& \text{ for } p< j\leq r.
\end{cases}
\]
In particular, the Waldschmidt constant is
\[
\widehat{\alpha}^p(I_\Sigma)=\frac{r \ \gamma_{p}(\Sigma)}{r-p+1}.
\]
\end{theorem}

		\begin{proof}
			Using \Cref{lemma: Symbolic power one diagram invariant}~(3) we have $\gamma_j(I_\Sigma)=\min\{\gamma_j(\sigma): \sigma\in \Sigma\}=\gamma_j(\Sigma)$,  and moreover
				\[\gamma_j\left(I_\Sigma^{(m)_p}\right )=\min\left \{ \gamma_j(\l_{\geq p}) : \lambda \in \Lambda_m(\Sigma) \right \}.\]

				If $p < j \leq r$, then we use \Cref{prop: sigma'}~(2) to compute
\[
\gamma_j(I_\Sigma^{(m)_p})=\gamma_{j-p+1}(I_{\Sigma_{\leftarrow p}}^{(m)_1})=\gamma_{j-p+1}(I_{\Sigma_{\leftarrow p}}^{m})=m\gamma_{j-p+1}(I_{\Sigma_{\leftarrow p}})=m\gamma_j(I_\Sigma).
\]
concluding that $\widehat{\gamma_j}^p(I_\Sigma)=\gamma_j(I_\Sigma)=\gamma_j(\Sigma)$.

Assume now that $1\leq j\leq p$. We establish the desired equality by proving two inequalities.

\medskip
\noindent
\textbf{\em Upper bound.}
Fix a positive integer $j$, set $q=r-p+1$ and let $\sigma\in\Sigma$ be a diagram that satisfies $\gamma_p(\s)= \gamma_p(\Sigma)$. By \Cref{lemma: Symbolic power one diagram invariant}~(2) the $q$-th generalized symbolic power $I_\Sigma^{(q)_p}$ contains $I_{\l}$ for any  $\l\in \Lambda_q(\s)_{\geq p}$. We will show that there exists a diagram $\l\in \Lambda_q(\s)_{\geq p}$ so that $\gamma_j\left(\l\right)=(r-j+1)\gamma_p(\s)$.

To this end we first find a rectangular diagram $\l'\in \Lambda_q(\s_{\leftarrow p})$. Set $\ell=|\s_{\leftarrow p}|=\gamma_p(\s)$. By \eqref{eq: Lambda} we have
\[
\Lambda_q(\s_{\leftarrow p}) =\left\{\l\in H_{r-p+1} : c^{\l^\vee}_{\underbrace{\s_{\leftarrow p}^\vee,\s_{\leftarrow p}^\vee, \ldots, \s_{\leftarrow p}^\vee}_{q}} \neq 0\right\},
\]
but since $\s_{\leftarrow p}\in H_{r-p+1}=H_q$ \Cref{thm: rectangle corrected}
 gives that $\l'=(q^{\ell})\in \Lambda_q(\s_{\leftarrow p})$  and is a minimal element of $\Lambda_q(\s_{\leftarrow p})$ as can be seen by considering its total number of boxes.

 By \Cref{prop: sigma'}~(1)  with $t=p$ we have $\l:=\l_{\rightarrow p}'\in \Lambda_q(\s)_{\geq p}$, where $\l$  is the rectangular diagram $\l=(r,\ldots, r)=(r^\ell)$ with $\ell=\gamma_p(\s)$ rows of length $r$. It satisfies $\gamma_j\left(\l\right)=(r-j+1)\ell=(r-j+1)\gamma_p(\s)$ thus by  \Cref{lemma: Symbolic power one diagram invariant}~(2) we conclude
\[
\gamma_j(I_\Sigma^{(q)_p}) \le (r-j+1) \gamma_p(\s)=(r-j+1) \gamma_p(\Sigma)\text{ for all } 1\leq j\leq p.
\]
Since the generalized symbolic powers form a graded family by \Cref{lem: generalized symbolic powers graded},
$\left(I_\Sigma^{(q)_p}\right)^k\subseteq I_\Sigma^{(qk)_p}$
for all $k\geq 1$. Therefore
\[
\gamma_j \left(I_\Sigma^{(qk)_p}\right) \le \gamma_j \left(\left(I_\Sigma^{(q)_p}\right)^k\right)=k \gamma_j \left(I_\Sigma^{(q)_p}\right)\leq  k(r-j+1)\gamma_p(\Sigma),
\]
The limit defining $\widehat{\gamma_j}^p(I_\Sigma)$ exists by subadditivity and Fekete's lemma, and hence it agrees with the limit along the subsequence of multiples of $q=r-p+1$. We conclude
\begin{equation}\label{eq:upper}
\widehat{\gamma_j}^p(I_\Sigma)=\lim_{k\to\infty}\frac{\gamma_j(I_\Sigma^{(kq)_p})}{kq}\le \frac{(r-j+1) \gamma_p(\Sigma)}{q}
=\frac{(r-j+1)\gamma_p(\Sigma)}{r-p+1}.
\end{equation}

\medskip
\noindent
\textbf{\em Lower bound.}
Let $\lambda\in\Lambda_m(\Sigma)$ be a Young diagram.  Assign content to the  boxes of $\l$ as prescribed by the union of the Yamanouchi fillings of a Yamanouchi $m$-filtration as in \Cref{def: Yamanouchi filtration}. With respect  to this content of $\l$, at least $m\gamma_p(I_\Sigma)$ boxes have entries $\ge p$. By the Yamanouchi condition and \cite[Lemma~5.1]{WhiteheadThesis}, the first occurrence of an integer $i$ in a Yamanouchi word must appear at the bottom of some column indexed $k\ge i$. Hence all boxes whose entries are at least $p$ must lie  weakly to the right of column $p$.

Since each row  of $\l\in H_r$ has length at most $r$, there are at most $r-p+1$ boxes in each row weakly to the right of column $p$. At each filtration stage the number of entries at least $p$ is at least $\gamma_p(\Sigma)$ so the total number of such entries is at least $m\gamma_p(\Sigma)$. Therefore,
\begin{equation}\label{eq: rows}
\#\{\text{rows of }\lambda \text{ containing entries }\ge p\}
\ge \frac{m \gamma_p(\Sigma)}{r-p+1}.
\end{equation}

Each such row contains the $p-1$ boxes in columns $1,\ldots,p-1$, and hence has total length at least $p$. Consequently, all these rows remain present in $\lambda_{\geq p}$. Since $j\leq p$, the boxes in columns $j,\ldots,p-1$ contribute at least $p-j$ additional boxes in each such row to $\gamma_j(\lambda_{\geq p})$. These boxes are disjoint from the boxes with entries at least $p$, all of which lie weakly to the right of column $p$. Therefore
\[
\gamma_j\left(\l_{\geq p}\right )
\ge (p-j)\cdot\frac{m\gamma_p(\Sigma)}{r-p+1} +m\gamma_p(\Sigma)= \frac{(r-j+1)m\gamma_p(\Sigma)}{r-p+1}.
\]
Applying \Cref{lemma: Symbolic power one diagram invariant}~(3), dividing by $m$ and letting $m\to\infty$ yields
\begin{equation}\label{eq:lower}
\widehat{\gamma_j}^{\,p}\left(I_\Sigma\right )\ge \frac{(r-j+1)\gamma_p(\Sigma)}{r-p+1}.
\end{equation}

Combining \eqref{eq:upper} and \eqref{eq:lower} completes the proof.
	\end{proof}

Utilizing the skew Waldschmidt constants computed above we determine the asymptotic resurgence of a $G$-invariant ideal.

\begin{theorem}\label{theorem: resurgence}
Let $\Sigma\subset H_r$ be a set of Young diagrams and let $p\leq r$ be a positive integer such that $\gamma_p(\Sigma)>0$.
Assume every Rees valuation of $I_\Sigma$ is one of $\gamma_1,\ldots,\gamma_r$. Then the asymptotic resurgence for the graded family $I_\Sigma^{(\bullet)_p}$ relative to $I_\Sigma^\bullet$ is
\[
\widehat{\rho}(I_\Sigma^{(\bullet)_p}, I_\Sigma^\bullet)=\frac{\alpha(I_\Sigma)}{\widehat{\alpha}^p(I_\Sigma)}=\frac{(r-p+1)\gamma_1(\Sigma)}{r\gamma_p(\Sigma)}.
\]
\end{theorem}

\begin{proof}
We first show that
\begin{equation}\label{eq: gamma_less_than_alpha}
\frac{\gamma_j(I_\Sigma)}{\widehat{\gamma_j}^p(I_\Sigma)}\leq \frac{\alpha(I_\Sigma)}{\widehat{\alpha}^{p}(I_\Sigma)}
\quad \text{for every $j$}.
\end{equation}

If $1\leq j \leq p$ it follows from \Cref{theorem: waldschmidt constant one diagram} that \eqref{eq: gamma_less_than_alpha} is equivalent to
\[
\frac{(r-p+1)\gamma_j(\Sigma)}{(r-j+1)\gamma_{p}(\Sigma)} \leq \frac{(r-p+1)\gamma_1(\Sigma)}{r\gamma_{p}(\Sigma)}
\iff
\frac{\gamma_j(\Sigma)}{r-j+1}\leq\frac{\gamma_1(\Sigma)}{r}.
\]
Indeed, for any $\sigma=(\sigma_1,\dots,\sigma_s)\in\Sigma$  we have
\[
\frac{\gamma_j(\sigma)}{r-j+1}=\frac{\sum_{i}\max\{0,\s_i-j+1\}}{r-j+1}\le \frac{\sum_{i}\s_i }{r}=\frac{\gamma_1(\sigma)}{r}.
\]
Taking the minimum over $\sigma\in\Sigma$ yields the desired inequality
\[
\frac{\gamma_j(\Sigma)}{r-j+1}\le \frac{\gamma_1(\Sigma)}{r}.
\]

If $p < j \leq r$ and $\gamma_j$ is a Rees valuation of $I_\Sigma$, then $\gamma_j(I_\Sigma)>0$ and \Cref{theorem: waldschmidt constant one diagram} gives
\[
\frac{\gamma_j(I_\Sigma)}{\widehat{\gamma_j}^p(I_\Sigma)}=1.
\]
On the other hand, since $I_\Sigma^{m}\subseteq I_\Sigma^{(m)_p}$, we have $\alpha(I_\Sigma^{(m)_p})\leq \alpha(I_\Sigma^m)=m\alpha(I_\Sigma)$ for all $m\geq 1$. Consequently, $\widehat{\alpha}^p(I_\Sigma)\leq \alpha(I_\Sigma)$ and
$\alpha(I_\Sigma)/\widehat{\alpha}^p(I_\Sigma)\geq 1$. This establishes \eqref{eq: gamma_less_than_alpha} for every coordinate  valuation with $j>p$.

For the Rees valuation $\gamma_1$  of $I_\Sigma$ we have
\[
\frac{\gamma_1(I_\Sigma)}{\widehat{\gamma_1}^{\,p}(I_\Sigma)}
=\frac{\alpha(I_\Sigma)}{\widehat{\alpha}^{p}(I_\Sigma)}.
\]
The equation $x_1=\gamma_1(\Sigma)>0$ defines a  facet of $P_\Sigma+\mathbb R_{\geq0}^r$ not contained in a coordinate hyperplane. Thus the maximum in \Cref{thm: DFMS} is bounded above by $\alpha(I_\Sigma)/\widehat{\alpha}^{p}(I_\Sigma)$ by \eqref{eq: gamma_less_than_alpha}, and equality is attained by $\gamma_1$. This proves the claimed value of
$\widehat{\rho}(I_\Sigma^{(\bullet)_p},I_\Sigma^\bullet)$ with input from \Cref{theorem: waldschmidt constant one diagram}.
\end{proof}

\begin{remark}
The formulas in \Cref{theorem: waldschmidt constant one diagram} and \Cref{theorem: resurgence} can be specialized to the families $(I_\Sigma)_{\Ass}^{(m)}$ and $(I_\Sigma)_{\Min}^{(m)}$ by setting $p=\min\{\s_i : \s\in \Sigma\}$ and $p=\min\{\s_1 : \s\in \Sigma\}$, respectively, in view of \Cref{lem: symbolic_powers_maxass} and \Cref{lemma: symbolic power for sum of ideal invariant}.
\end{remark}

\begin{example}\label{ex: mixed Rees valuation counterexample}

The formula in \Cref{theorem: resurgence} need not hold for an arbitrary set of diagrams. Let
\[
r=3,\qquad p=2,\qquad
\Sigma=\{(3,3),(2,1^8)\}.
\]
The corresponding valuation vectors are
$\gamma(3,3)=(6,4,2)$,
$\gamma(2,1^8)=(10,1,0)$.
Thus $\alpha(I_\Sigma)=\gamma_1(\Sigma)=6$ and
$\gamma_2(\Sigma)=1$, so \Cref{theorem: waldschmidt constant one diagram} gives
\[
\frac{\alpha(I_\Sigma)}{\widehat{\alpha}^2(I_\Sigma)}
=\frac{2\gamma_1(\Sigma)}{3\gamma_2(\Sigma)}=4.
\]

The facets of $P_\Sigma+\RR_{\geq0}^3$ are the hyperplanes $x_1=6, x_2=1,x_3=0, 3x_1+4x_2=34, x_1+2x_3=10$.
Hence, by
\Cref{thm: Rees_valuations I_Sigma},
$
v=\gamma_1+2\gamma_3
$
is a Rees valuation of $I_\Sigma$, and $v(I_\Sigma)=10$.

On the other hand, since
\[
\Sigma_{\leftarrow2}=\{(2,2),(1)\},
\qquad
I_{\Sigma_{\leftarrow2}}=I_{(1)},
\]
is the homogeneous maximal ideal. By \Cref{prop: sigma'}~(3) with $t=2$ we conclude that  $I_\Sigma^{(m)_2}=\sum_{\lambda\in H_2, \, |\lambda|=m} I_{\l_{\rightarrow 2}}$ for all $m\geq 1$. If $\l=(1^m)$ then $ \l_{\rightarrow 2}=(2^m)$ and $v(\l_{\rightarrow 2})=2m$ and in general if $|\l|=m$ then $v(\l_{\rightarrow 2})=2|\l|+\gamma_2(\l)\geq 2m$, so
\[
v\bigl(I_\Sigma^{(m)_2}\bigr)= 2m \qquad\text{and hence}\qquad
\widehat v(I_\Sigma^{(\bullet)_2})= 2.
\]

Consequently, \Cref{thm: DFMS} gives
\[
\widehat{\rho}
 \bigl(I_\Sigma^{(\bullet)_2},I_\Sigma^\bullet\bigr)
\geq
\frac{v(I_\Sigma)}{\widehat v(I_\Sigma^{(\bullet)_2})}
= \frac{10}{2}=5>4
=\frac{\alpha(I_\Sigma)}{\widehat{\alpha}^2(I_\Sigma)}.
\]
Thus mixed Rees valuations can yield a larger ratio than the degree valuation.

\end{example}

As a corollary of \Cref{theorem: resurgence} we are able to infer asymptotic invariants for isotypic ideals $I_\s$ and products of determinantal ideals. Recall that as in \eqref{eq: Dsigma} we set $D_\sigma = D_{\sigma_1}D_{\sigma_2}\cdots D_{\sigma_s}$.

\begin{theorem}\label{thm: Ds}
Let $\s \in H_r$ be a diagram with $s$ rows and set
\[
\Lambda= \left\{ \l \in H_r : c_{(\s_1)^\vee,(\s_2)^\vee,\cdots ,(\s_s)^\vee}^{\l^\vee}\neq 0\right \}.
\]
Then

\begin{enumerate}
	\item $D_\sigma=\sum_{\tau \in \Lambda} I_\tau$,
	\item $\Min(D_\sigma)=\{D_{\s_1} \}$ and $\maxass(D_\sigma)=\{D_{t_0} \}$, where $t_0=\max\{1, r- \sum_{i=1}^s (r-\sigma_i)\}$,
	\item for any positive integer $1\leq p\leq \s_1$ the skew
Waldschmidt constants for the families of generalized symbolic powers $D_\s^{(\bullet)_p}$ are given by
\[
\widehat{\gamma_j}^p(D_\s):=\lim_{m\to\infty}\frac{\gamma_j\left (D_\s^{(m)_p}\right )}{m}=
\begin{cases}
\frac{(r-j+1) \ \gamma_{p}(\s)}{r-p+1} & \text{ for } 1\leq j\leq p\\
\gamma_j(\s)& \text{ for } p< j\leq r.
\end{cases}
\]
The resurgence and asymptotic resurgence for   $D_\s^{(\bullet)_p}$ and the asymptotic resurgence for  $I_\s^{(\bullet)_p}$, respectively, relative to the corresponding ordinary powers is
\begin{eqnarray*}
{\rho}(D_\s^{(\bullet)_p}, D_\s^\bullet)=\widehat{\rho}(D_\s^{(\bullet)_p}, D_\s^\bullet)=\frac{\alpha(D_\s)}{\widehat{\alpha}^p(D_\s)}=\frac{(r-p+1)|\s|}{r\gamma_p(\s)}, \\
\widehat{\rho}(I_\s^{(\bullet)_p}, I_\s^\bullet)=\frac{\alpha(I_\s)}{\widehat{\alpha}^p(I_\s)}=\frac{(r-p+1)|\s|}{r\gamma_p(\s)}.
\end{eqnarray*}
If $\s_1<p\leq r$, then $I_\s^{(m)_p}=D_\s^{(m)_p}=R$ for every $m\geq1$, and
\[
\rho(I_\s^{(\bullet)_p}, I_\s^\bullet)=\widehat{\rho}(I_\s^{(\bullet)_p}, I_\s^\bullet)=\rho(D_\s^{(\bullet)_p},D_\s^\bullet)
=\widehat{\rho}(D_\s^{(\bullet)_p},D_\s^\bullet)=\infty.
\]
\end{enumerate}
	\end{theorem}
	\begin{proof}

	Assertion (1) follows by iterating \Cref{theorem: product whitehead} to obtain \[
D_\sigma=\prod_{i=1}^s D_{\sigma_i} =\prod_{i=1}^s I_{(\sigma_i)}=\sum_{\tau \in \Lambda} I_{\tau}.
\]

For (2)  set  $t_0=\min\{j: D_j\in \Ass(D_\s)\}$, $t_1=\max\{j: D_j\in \Ass(D_\s)\}$.

 In view of \cite[Theorem 7.2]{deconcini1980young} one has
\begin{equation}\label{eq: ass Ds}
\Ass(D_\s)=\left \{D_j: \s^\vee_j\neq 0 \text{ and } \gamma_{j+1}(\s)\leq (r-j)(\s^\vee_j-1)\right\}.
\end{equation}
The claim follows by
rewriting the condition $\gamma_{j+1}(\s)\leq (r-j)(\s^\vee_j-1)$ in terms of $\sum_i\sigma_i$ as detailed below. First we show that $t_0 \leqslant \sigma_s$.  Since $\s^\vee_{\s_s}=s$ and $\sigma_i\leq r$ for each $i$ we have
\[
\gamma_{\sigma_s+1}(\sigma)=\sum_{i=1}^s \max \{0, \sigma_i -(\sigma_s+1)+1\}=\sum_{i=1}^s(\sigma_i-\sigma_s) \leqslant (r-\sigma_s)(s-1)=(r-\sigma_s)(\s^\vee_{\s_s}-1).
\]
Thus  $D_{\sigma_s} \in \Ass(D_{\sigma})$ and hence $t_0 \leqslant \sigma_s$.
		Now for any $j \leq \sigma_s$, we have $\s^\vee_{j}=s$ and thus by \eqref{eq: ass Ds}
		\begin{align*}
			D_j \in \Ass(D_\sigma) &\iff \gamma_{j+1}(\sigma) \leqslant (r-j)(\s^\vee_j-1)\\
			& \iff \gamma_{j+1}(\sigma) \leqslant (r-j)(s-1), \qquad \text{since } \sigma_i \geqslant \sigma_s \geq j \\
			& \iff  \sum_{i=1}^{s}\max\{0, \sigma_i-(j+1)+1 \}\leqslant (r-j)(s-1)\\
			& \iff \sum_{i=1}^{s}(\sigma_i-j) \leqslant (r-j)(s-1)\\
			&\iff \sum_{i=1}^s\sigma_i -js \leqslant (r-j)(s-1)\\
			&\iff \sum_{i=1}^s\sigma_i  \leqslant r(s-1)+j\\
			&\iff \sum_{i=1}^s\sigma_i - r(s-1) \leqslant j.
\end{align*}
It follows that $t_0=\max\{1, r- \sum_{i=1}^s (r-\sigma_i)\}$.
 Since $\s^\vee_j=0$ for $j>\s_1$, \eqref{eq: ass Ds} gives  $t_1\leq \s_1$. Moreover, since $\gamma_{s_1+1}(\s)=0\leq (r-\s_1)(\s^\vee_{s_1}-1)$, it follows that $D_{\s_1}\in \Ass(D_\s)$ and thus $t_1=\s_1$.

For (3), recall from part~(1) that $D_\sigma=I_\Lambda$. Since
$
D_\sigma=\overline{I_\sigma},
$
taking integral closure does not change the value of an ideal under a valuation. Consequently, for every $1\leq j\leq r$,
\[
\gamma_j(\Lambda)
=\gamma_j(I_\Lambda)
=\gamma_j(D_\sigma)
=\gamma_j(I_\sigma)
=\gamma_j(\sigma).
\]
Therefore \Cref{theorem: waldschmidt constant one diagram}, applied to $I_\Lambda=D_\sigma$, gives the asserted skew Waldschmidt constants.

It remains to verify the additional hypothesis in \Cref{theorem: resurgence}. Since $D_\sigma=\overline{I_\sigma}$, the ideals $D_\sigma$ and $I_\sigma$ have the same Rees valuations. The polyhedron associated to the singleton $\{\sigma\}$ is
\[
P_{\{\sigma\}}+\RR_{\geq0}^r
=\gamma(\sigma)+\RR_{\geq0}^r.
\]
Its supporting facets with positive constant are coordinate facets, and hence the Rees valuations of $D_\sigma$ are the coordinate valuations $\gamma_j$. Thus \Cref{theorem: resurgence} applies and gives
\[
\widehat{\rho}(D_\sigma^{(\bullet)_p},D_\sigma^\bullet)
=\frac{\alpha(D_\sigma)}{\widehat{\alpha}^p(D_\sigma)}
=\frac{(r-p+1)|\sigma|}{r\gamma_p(\sigma)}
\quad \text{and} \quad
\widehat{\rho}(I_\sigma^{(\bullet)_p},I_\sigma^\bullet)
=\frac{\alpha(I_\sigma)}{\widehat{\alpha}^p(I_\sigma)}
=\frac{(r-p+1)|\sigma|}{r\gamma_p(\sigma)}.
\]

Finally, every power of $D_\sigma$ is integrally closed. Indeed, if $\sigma^{\sqcup m}$ denotes the diagram obtained by repeating every part of $\sigma$ exactly $m$ times, then
$
D_\sigma^m=D_{\sigma^{\sqcup m}}
=\overline{I_{\sigma^{\sqcup m}}}$.
It follows that $D_\sigma^\bullet=\overline{D_\sigma^\bullet}$. Therefore \Cref{thm: DFMS} yields
\[
\rho(D_\sigma^{(\bullet)_p},D_\sigma^\bullet)
=\widehat{\rho}(D_\sigma^{(\bullet)_p},D_\sigma^\bullet),
\]
for $1\leq p\leq\sigma_1$.

if instead $\sigma_1<p\leq r$, then $D_\sigma\not\subseteq D_p$, and hence
\[
D_\sigma^{(m)_p}=D_\sigma^mR_{D_p}\cap R=R
\qquad\text{for every }m\geq1.
\]
Since $D_\sigma^v$ is a proper ideal for every $v\geq1$, the defining noncontainments occur for arbitrarily large ratios, so both the resurgence and the asymptotic resurgence are $+\infty$. One argues similarly for $I_\s$.
		\end{proof}

\section{Generalized Symbolic Rees Algebras}\label{s: symbolic Rees}

For a $G$-invariant ideal $I\subseteq R$ and an integer $1\leq t\leq r$, we define its \emph{generalized symbolic Rees algebra with respect to $D_t$} by
\[
\mathcal R_s^t(I)
   :=\bigoplus_{m\geq0} I^{(m)_t}T^m\subseteq R[T],
\qquad I^{(0)_t}:=R.
\]
This is an $R$-algebra because the family $I^{(\bullet)_t}$ is graded by \Cref{lem: generalized symbolic powers graded}. We now prove that these algebras are Noetherian for every $G$-invariant ideal considered in this paper.

\begin{theorem}\label{thm: symbolic Rees Sigma}
Let $\Sigma\subset H_r$ be a finite nonempty set of Young diagrams and let $1\leq t\leq r$. Then the generalized symbolic Rees algebra
\[
\mathcal R_s^t(I_\Sigma)=\bigoplus_{m\geq0}I_\Sigma^{(m)_t}T^m
\]
is  finitely generated, and hence a Noetherian, as an $R$-algebra.
\end{theorem}

\begin{proof}
Finite generation descends through faithfully flat field extension \cite{StacksProject}, so it suffices to prove the claim over a  field extension of $\kk$. Therefore we may assume $\kk$ is algebraically closed.

If $I_\Sigma\not\subseteq D_t$, then $I_\Sigma^{(m)_t}=R$ for every $m\geq1$, so that $\mathcal R_s^t(I_\Sigma)=R[T]$. We may therefore assume that $I_\Sigma\subseteq D_t$.

Set $q=r-t+1$, $n'=n-t+1$, let $Y$ be a generic $q\times n'$ matrix, and let
$
G_Y=\GL_q(\kk )\times \GL_{n'}(\kk)
$
act naturally on $\kk[Y]$.
Define
\[
\Omega:=\Sigma_{\leftarrow t}
=\bigl\{\sigma_{\leftarrow t}:\sigma\in\Sigma\bigr\}\subset H_q.
\]
 Truncation may destroy the antichain property; in that case we
replace $\Omega$ by its containment-minimal elements, which does not change
the ideal indexed by $\Omega$, and we continue to use the same notation.

Set $J=I_\Omega(Y)\subseteq\kk[Y]$.
The reason for introducing this
smaller ideal is that \Cref{prop: sigma'}~(3) recovers the generalized symbolic
powers of $I_\Sigma$ from the ordinary powers of $J$.
The ordinary Rees algebra
\[
\mathcal R(J)=\bigoplus_{m\geq0}J^mT^m \subseteq \kk[Y][T]
\]
is a finitely generated bigraded $G_Y$-algebra where  $G_Y$ acts on $\kk[Y][T]$ through its natural action on  $\kk[Y]$ extended to $\kk[Y][T]$ by  fixing $T$.  Let $U_Y\subset G_Y$ be a maximal unipotent subgroup, for example, $U_Y=U^-_q\times U^+_{n'}$ where $U^+_i, U^-_i$ denote, respectively,  the upper and lower triangular groups of $i\times i$ matrices with 1's on the main diagonal.
Then Grosshans' finite-generation theorem
\cite{Grosshans1983}  gives that
$\mathcal R(J)^{U_Y}$ is a finitely generated algebra.

Each simple module $M_\lambda$ in the Cauchy decomposition of $\kk[Y]$ has a one-dimensional highest-weight space which coincides with its $U_Y$-invariant submodule by \cite[Theorem 2.4]{Brion2010}. In fact, as the polynomial $K_\l$ in \eqref{eq: K} is $U_Y$-invariant we conclude that $(M_\lambda)^{U_Y}=\operatorname{span}_{\kk}(K_\l)$.
If $\s,\lambda, \mu\in H_q$ satisfy $\sigma^\vee=\lambda^\vee+\mu^\vee$, then $\sigma$ consists of the rows of $\lambda$ together with the rows of $\mu$, ordered weakly decreasingly. Therefore the identity $K_\l K_\mu=K_\s$ holds and it leads to the conclusion
\[(M_\l )^{U_Y} ( M_\mu)^{U_Y}=\operatorname{span}_{\kk}(K_\l K_\mu)=\operatorname{span}_{\kk}(K_\s)=(M_\s)^{U_Y}.\]
Equivalently the decomposition $\kk[Y]^{U_Y}=\bigoplus_{\s\in H_q} M_\s^{U_Y}=\bigoplus_{\s\in H_q} \operatorname{span}_{\kk}(K_\s)$ makes   $\kk[Y]^{U_Y}$ into a $H_q^\vee$-graded ring by assigning weight $\s^\vee$ to $K_\s$.

We extend this to a bigrading on $(\kk[Y][T])^{U_Y}=\kk[Y]^{U_Y}[T]$ where the elements of  $(M_\lambda)^{U_Y}T^m$ are homogeneous of weight $(m,\l^\vee)$. As before  $(M_\lambda)^{U_Y}T^m$ is  a one-dimensional vector space spanned by a highest weight vector $v_{(m,\l)}$.
The subring $\mathcal R(J)^{U_Y}\subseteq (\kk[Y][T])^{U_Y}$ is thus graded by the following set of weights, which forms an affine semigroup
\begin{equation}\label{eq: S(J)}
S(J):=\bigl\{(m,\lambda^\vee): M_\lambda \subseteq J^m\bigr\} \subseteq \ZZ_{\geq0}^{q+1}.
\end{equation}
Here a conjugate partition $\lambda^\vee$ is regarded
as a  $q$-tuple by appending zeros.

Since $\mathcal R(J)^{U_Y}$ has  a finite set of generators, and is graded with respect to the weights \eqref{eq: S(J)}, it also has a finite set of homogeneous generators with respect to these weights, e.g.  the homogeneous components of any finite generating set. Let $v_{(m_1,\l^{(1)})}, v_{(m_2,\l^{(2)})}, \cdots v_{(m_N,\l^{(N)})}$ be algebra generators for $\mathcal R(J)^{U_Y}$, which can be assumed highest-weight vectors by the preceding observation. Let $f\in \mathcal R(J)^{U_Y}$ have weight $(m, \lambda^\vee)$. Then $f$ can be written up to a scalar as a monomial in  $v_{(m_1,\l^{(1)})}, v_{(m_2,\l^{(2)})}, \cdots v_{(m_N,\l^{(N)})}$ as the vector space of elements of weight $(m, \lambda^\vee)$ is one dimensional. Consequently $(m, \lambda^\vee)$ can be written as a $\NN$-linear combination of $(m_1,{\l^{(1)}}^\vee), (m_2,{\l^{(2)}}^\vee), \cdots,$ $(m_N,{\l^{(N)}}^\vee)$, showing that these weights generate the semigroup $S(J)$.

Define the additive map
\[
L_t:\ZZ_{\geq0}^{q}\longrightarrow\ZZ_{\geq0}^{r},
\qquad
L_t(a_1,\ldots,a_q)
   =(\underbrace{a_1,\ldots,a_1}_{t\text{ entries}},a_2,\ldots,a_q).
\]
For every $\lambda\in H_q$, the definition of the shift operation gives
\begin{equation}\label{eq: conjugate shift}
(\lambda_{\rightarrow t})^\vee=L_t(\lambda^\vee).
\end{equation}

Let  $U\subseteq G=\GL_r(\kk)\times\GL_n(\kk)$ be a maximal unipotent subgroup corresponding to the chosen Borel subgroup.
The same argument shows that $\mathcal R_s^t(I_\Sigma)^{U}$
is graded by the weight semigroup  \[
S_t(I_\Sigma):=\left\{(m,\lambda^\vee): M_\lambda \subseteq I_\Sigma^{(m)_t}\right\} \subseteq \ZZ_{\geq 0}^{r+1}.
\]
The formula in \Cref{prop: sigma'}~(3), together with \Cref{rem: G-invariant}~(1) and \eqref{eq: conjugate shift}, gives the following alternate description of this semigroup
\begin{equation}\label{eq: symbolic Rees weight semigroup}
S_t(I_\Sigma)
=\left\{(m,\beta):
\begin{array}{l}
\beta\in H_r^\vee \text{ and there is }(m,\alpha)\in S(J)\\[-2pt]
\text{such that }\beta_i\geq L_t(\alpha)_i\text{ for }1\leq i\leq r
\end{array}
\right\}.
\end{equation}
Indeed, $M_\lambda \subseteq J^m$ means that $\lambda$ contains a diagram in $\Lambda_m(\Omega)$. Moreover  $M_\tau \subseteq I_\Sigma^{(m)_t}$ if and only if $\tau$ contains $\lambda_{\rightarrow t}$ for some  $\lambda \in \Lambda_m(\Omega)$, which, after conjugating gives $\tau^\vee_i\geq (\l_{\rightarrow t}^\vee)_i$ for $1\leq i\leq r$  and using \eqref{eq: conjugate shift} becomes  \eqref{eq: symbolic Rees weight semigroup}.

Recall that $(m_1,{\l^{(1)}}^\vee), (m_2,{\l^{(2)}}^\vee), \cdots, (m_N,{\l^{(N)}}^\vee)$ are generators for $S(J)$. Consider the semigroup
\[
\mathcal{P}=\left\{ (c_1,\ldots,c_N,\beta)\in \ZZ^{N+r} : \
c_i\geq0, \ \beta_1\geq\cdots\geq\beta_r\geq0, \
\beta_j\geq\sum_{i=1}^N c_iL_t({\l^{(i)}}^\vee)_j \text{ for }1\leq j\leq r \right \}.
\]
It is the set of lattice points of a rational polyhedral cone, hence is finitely generated by Gordan's lemma. The image of $\mathcal{P}$ under the linear transformation
\[
(c_1,\ldots,c_N,\beta)
\longmapsto
\left(\sum_{i=1}^Nc_im_i,\beta\right)
\]
is exactly the semigroup $S_t(I_\Sigma)$ in \eqref{eq: symbolic Rees weight semigroup}. Therefore $S_t(I_\Sigma)$ is finitely generated.

Finally, let $\delta_1,\ldots,\delta_s$ with $\delta_i=(m_i,\beta^{(i)})$ be a finite generating set for the semigroup $S_t(I_\Sigma)$.  Set  $w_i=K_{{\beta^{(i)}}^\vee}T^{m_i}$ for each member of the previous set. As
\[
V_i:=\operatorname{span}_{\kk}(G\cdot w_i)= M_{{\beta^{(i)}}^\vee}T^{m_i}
\]
and the latter is a finite dimensional vector space, one can pick a finite basis $\mathcal B_i$ of $V_i$ and let
\[ C = R\bigl[\mathcal B_1\cup\cdots\cup\mathcal B_s\bigr] \subseteq \mathcal R_s^t(I_\Sigma). \]
Note that $C$ is stable under the action of $G$ on $\mathcal R_s^t(I_\Sigma)$ and contains $w_1, \ldots, w_s$.

Since every highest-weight vector in $\mathcal R_s^t(I_\Sigma)$ is an element of $\mathcal R_s^t(I_\Sigma)^U$, up to a nonzero scalar, it is a product of the elements $w_i$. Thus all highest-weight vectors in $\mathcal R_s^t(I_\Sigma)$ are elements of $C$. Now let $M_\lambda T^m$ be any irreducible $G$-submodule occurring in $\mathcal R_s^t(I_\Sigma)$. Its highest-weight vector lies in $C$, and since $M_\lambda T^m$ is spanned by the $G$-orbit of its highest-weight vector and $C$ is $G$-stable, the entire module $M_\lambda T^m$ lies in $C$. This gives the equality
$C=\mathcal R_s^t(I_\Sigma)$, which proves finite generation.
\end{proof}

\begin{appendix}

\section{Nonvanishing of some top powers of Schubert classes}\label{appendix}

This section gives a self-contained, combinatorial explanation for the nonvanishing of certain powers of Schubert classes.
The main theorem,  \Cref{thm: rectangle corrected}, is used in the proof of the upper bound in \Cref{theorem: waldschmidt constant one diagram}.

Before we state and prove the main theorem by a rather technical construction, we illustrate the idea by means of an example.

\begin{example}
Set  $r=5$ and $\s=(4,3,1)$ and notice $|\s|=8$. We wish to show that $c_{\underbrace{\scriptstyle\sigma^\vee,\ldots,\sigma^\vee}_{5}}^{\rho^\vee}> 0$ where $\rho=(5^8)$. It turns out, as shall be explained in the proof of  \Cref{thm: rectangle corrected} that this is equivalent to constructing a Yamanouchi filling for a diagram consisting of 5 copies of $\s$ placed in disjoint rows and columns  with filling given by the multiset $\rho^\vee=\{1^8, 2^8, 3^8, 4^8, 5^8\}$.

To construct such a filling we consider an auxiliary diagram $\theta$ as in \Cref{fig: theta}, which one can think of as the superposition of $r$ shifted copies of $\sigma$, the $i$-th copy being filled constantly with $i$ and shifted to have its leftmost column in column $i$ of $\theta$. The diagram $\theta$ is then filled with sets of integers denoted $S_{i,j}$ in the proof below, where $i$ labels the row and $j$ the column. Each set $S_{i,j}$ is shown in box $(i,j)$ of $\theta$ in \Cref{fig: theta} below.

\begin{figure}[ht]
 \scalebox{.9}{
\ytableausetup{boxsize=4em}
\begin{ytableau}
       1 & 1,2  & 1,2,3 & 1,2,3,4 & 2,3,4,5 & 3,4,5 & 4,5 & 5 \\
     1 & 1,2  & 1,2,3 & 2,3,4 & 3,4,5 & 4,5 & 5 &  \none \\
           1 & 2  & 3 & 4 & 5 &\none & \none & \none \\
\end{ytableau}
}
\caption{The diagram $\theta$ with its filling (the set $S_{ij}$ appears in row $i$, column $j$).}
\label{fig: theta}
\end{figure}

\begin{figure}[ht]
    \centering
    \ytableausetup{boxsize=1.7em}
    \setlength{\tabcolsep}{8pt}

    \begin{tabular}{ccccc}

        \begin{ytableau}
            1 & 2 & 3 & 4 \\
            1 & 2 & 3 \\
            1
        \end{ytableau}
         &
        \begin{ytableau}
            1 & 2 & 3 & 5 \\
            1 & 2 & 4 \\
            2
        \end{ytableau}
          &
        \begin{ytableau}
            1 & 2 & 4 & 5 \\
            1 & 3 & 5 \\
            3
        \end{ytableau}
         &
        \begin{ytableau}
            1 & 3 & 4 & 5 \\
            2 & 4 & 5 \\
            4
        \end{ytableau}
        &
        \begin{ytableau}
            2 & 3 & 4 & 5 \\
            3 & 4 & 5 \\
            5
        \end{ytableau}
        \\[4pt]

        $\nu^{(1)}$
        &
        $\nu^{(2)}$
        &
        $\nu^{(3)}$
        &
        $\nu^{(4)}$
        &
        $\nu^{(5)}$
    \end{tabular}
    \caption{The five tableaux of shape $\sigma=(4,3,1)$,
    arranged from left to right as
    $\nu^{(1)}\sqcup\nu^{(2)}\sqcup\nu^{(3)}
    \sqcup\nu^{(4)}\sqcup\nu^{(5)}$, the order used for
    the column-reading word.}
    \label{fig:filling}
\end{figure}

We will construct 5 fillings of $\s$ based on $\theta$ and called  $\nu^{(1)},\nu^{(2)},\nu^{(3)},
\nu^{(4)},\nu^{(5)}$. One way to construct them is iterative, starting with $\nu^{(1)}$, continuing with $\nu^{(2)}$, etc. We visualize the filling of the diagram $\nu^{(j)}$ as arising from the  shape $\s$ placed inside $\theta$ with its leftmost column in column $j$. For each of the boxes in this shape we pick out the largest entry available in that position after the entries used up by the previous fillings were deleted and we place that entry in the corresponding box of $\nu^{(j)}$. The result of this algorithm is presented in \Cref{fig:filling}.

In the proof of \Cref{thm: rectangle corrected} we give a deterministic (non-recursive) formula for the fillings $\nu^{(j)}$ and we show that $\nu^{(1)}\sqcup\nu^{(2)}\sqcup\nu^{(3)}
    \sqcup\nu^{(4)}\sqcup\nu^{(5)}$ is a Yamanouchi filling.
\end{example}

\begin{theorem}\label{thm: rectangle corrected}
Let $\sigma=(\sigma_1,\ldots,\sigma_s)\in H_r$, let
$\ell=|\sigma|$, and set $\rho=(r^\ell)$.  Then
\[
c_{\underbrace{\scriptstyle\sigma^\vee,\ldots,\sigma^\vee}_{r\text{ copies}}}^{\rho^\vee}>0.
\]
\end{theorem}

\begin{proof}
Set $a=\sigma_1$.  After omitting trailing zero parts, define partitions
$\widetilde\mu\subseteq\widetilde\lambda$ by
\begin{equation}\label{eq: corrected tilde partitions}
 \widetilde\mu_{qs+i}=(r-1-q)a,
 \qquad
 \widetilde\lambda_{qs+i}=(r-1-q)a+\sigma_i
 \quad
 (0\leq q\leq r-1,\ 1\leq i\leq s).
\end{equation}
These are partitions: within each block this follows from the fact that
$\sigma$ is a partition, while at the boundary between two consecutive blocks
the last part of the upper block is strictly larger than the first part of the
lower block.  The skew diagram $\widetilde\lambda/\widetilde\mu$ is the disjoint
union of $r$ translates $\nu^{(1)},\ldots,\nu^{(r)}$ of $\sigma$.  We label them
from left to right, so that the box $(i,j)$ of $\nu^{(k)}$ has global (row,column)--coordinates
\[
 ((r-k)s+i,(k-1)a+j).
\]
In particular, distinct components occupy disjoint sets of rows and columns.
Consequently their fillings are independent and by  \cite[Proposition 9]{Zelevinsky} we have
\begin{equation}\label{eq: corrected multiple to ordinary LR}
c_{\widetilde\mu^\vee,\rho^\vee}^{\widetilde\lambda^\vee}
=c_{\underbrace{\scriptstyle\sigma^\vee,\ldots,\sigma^\vee}_{r\text{ copies}}}^{\rho^\vee}.
\end{equation}

It remains to construct a Yamanouchi filling of
$\widetilde\lambda/\widetilde\mu$ with content
$\rho^\vee=(\ell^r)$.  Put
\[
\theta=(\sigma_1+r-1,\ldots,\sigma_s+r-1)
\]
and, for a box $(i,c)$ of $\theta$, set
\begin{equation}\label{eq: corrected S sets}
L_i(c)=\max\{1,c-\sigma_i+1\},\qquad
U(c)=\min\{c,r\},\qquad
S_{i,c}=[L_i(c),U(c)]\cap\mathbb Z.
\end{equation}
Thus $v\in S_{i,c}$ if and only if
$v\leq c\leq v+\sigma_i-1$ and $v\leq r$. Equivalently the sets $S_{i,c}$ arise from combining the fillings  of $r$ shifted copies of $\s$ where the $k$-th copy is shifted $k-1$ columns to the right and filled constantly by $k$. It contributes $k$ exactly once for each of the $\ell$ boxes of $\s$. Therefore each $k\in[r]$ has multiplicity $\ell$ and it follows that, as multisets,
\begin{equation}\label{eq: corrected content}
\mathop{\biguplus}_{(i,c)\in\theta}S_{i,c}
=\{1^\ell,2^\ell,\ldots,r^\ell\}.
\end{equation}

For each occurrence $v\in S_{i,c}$, define
\begin{equation}\label{eq: corrected stage assignment}
\kappa_i(c,v)=L_i(c)+U(c)-v,
\qquad j=c-\kappa_i(c,v)+1,
\end{equation}
and place the entry $v$ in the box $(i,j)$ of the component
$\nu^{(\kappa_i(c,v))}$.  Observe that  $1\leq \kappa_i(c,v)\leq U(c)\leq r$ and since $v$ belongs to
$S_{i,c}$, an elementary calculation shows that $1\leq j\leq\sigma_i$. Thus the placement makes sense.

 Conversely, for every box
$(i,j)$ of $\nu^{(k)}$, take $c=j+k-1$.  Then $k\in S_{i,c}$, and there is
exactly one occurrence assigned to that box. Its value is
\begin{equation}\label{eq: corrected nu formula}
\nu_{i,j}^{(k)}=L_i(c)+U(c)-k=
\max\{1,j+k-\sigma_i\}+\min\{j+k-1,r\}-k.
\end{equation}
Thus the assignment is well defined on every box and, by
\eqref{eq: corrected content}, has the required content.

We next verify standardness.  In a fixed row of $\nu^{(k)}$, both terms in
\eqref{eq: corrected nu formula} are nondecreasing as $j$ increases.  They
cannot both remain constant from $j$ to $j+1$: constancy of the first would require
$j+k-1\leq\sigma_i-1$ and $\sigma_i\leq r$, whereas constancy of the second would
require $j+k-1\geq r$. Hence rows are strictly
increasing.  Since $\s_i\geq\sigma_{i+1}$, we have
\[
\max\{1,j+k-\sigma_i\}
\leq\max\{1,j+k-\sigma_{i+1}\},
\]
so columns are weakly increasing from top to bottom.  Since distinct
components have disjoint rows and columns, the entire skew filling is
standard.

Finally, we prove the Yamanouchi condition.  For $u<v$ and an occurrence
$v_{i,c}$ of $v$ coming from $S_{i,c}$, define
\begin{equation}\label{eq: corrected pairing}
f_{u,v}(v_{i,c})=u_{i,c-v+u}.
\end{equation}
The equivalences
\begin{eqnarray*}
v\in S_{i,c} \quad & \Longleftrightarrow & \quad v\leq c\leq v+\sigma_i-1\\
u\in S_{i,c-v+u} \quad & \Longleftrightarrow & \quad u\leq c-v+u\leq u+\sigma_i-1
\end{eqnarray*}
show that this is a bijection from the occurrences of $v$ to those of $u$.
Let $g_i(c)=L_i(c)+U(c)$.  Each increment
$g_i(c+1)-g_i(c)$ is at least one.  Indeed, the increment of $U$ can vanish
only when $c\geq r$, and that of $L_i$ can vanish only when
$c\leq\sigma_i-1$; these conditions cannot hold simultaneously because
$\sigma_i\leq r$.

If $v_{i,c}$ is placed in component $\nu^{(k)}$ and its paired occurrence
is placed in $\nu^{(t)}$, then
\[
k=g_i(c)-v,
\qquad
t=g_i(c-v+u)-u.
\]
Summing the preceding increment inequality over $v-u$ steps gives $t\leq k$.
If $t<k$, the component containing the paired $u$ lies strictly to the left
of the component containing $v$, so the $u$ occurs first in the column-reading
word.  If $t=k$, the two entries lie in the same row of the same component,
and the local column of $u$ is smaller by $v-u$, so again the $u$ occurs
first.  Therefore, in every prefix of the column-reading word of $\widetilde\lambda/\widetilde\mu=\nu^{(1)}\sqcup\cdots \sqcup \nu^{(r)}$, the number of
$u$'s is at least the number of $v$'s for every $u<v$.  In other words, the filling is
Yamanouchi.  The dual Littlewood--Richardson rule now gives
$c_{\widetilde\mu^\vee,\rho^\vee}^{\widetilde\lambda^\vee}>0$, and the conclusion follows
from \eqref{eq: corrected multiple to ordinary LR}.
\end{proof}

\begin{corollary}\label{cor: Schubert corrected}
Let $k,r$ be positive integers, and let $\mathcal S_\sigma$ be the Schubert
class indexed by a partition $\sigma$ satisfying $|\sigma|\leq k$ and
$\sigma_1\leq r$ in $H^*(\Gr(k,k+r),\mathbb Z)$.  Then
$\mathcal S_\sigma^r\neq0$.  If $|\sigma|=k$, then
\[
\mathcal S_\sigma^r=c[\mathrm{pt}]
\qquad\text{for some integer }c>0.
\]
\end{corollary}

\begin{proof}
Set $\ell=|\sigma|$.  The rectangle $(r^\ell)$ fits inside the $k\times r$
rectangle because $\ell\leq k$.  By \Cref{thm: rectangle corrected}, its
coefficient in $\mathcal S_\sigma^r$ is positive, proving nonvanishing.  If
$\ell=k$, this rectangle is the full $k\times r$ rectangle and represents the
point class.  Since the product already has top degree, the displayed formula
follows.  The last assertion records the distinction between a nonzero
multiple of the generator of the rank-one group
$H^{2kr}(\Gr(k,k+r),\mathbb Z)$ and an integral generator.
\end{proof}

If $|\s|=k$, \Cref{cor: Schubert corrected} ensures that $\mathcal S_\sigma^r$ spans top cohomology after tensoring with $\mathbb Q$, but it need not
generate the top integral cohomology group.

\begin{example}\label{ex: corrected integral comparison}
For $\sigma=(2,1)$  in $\Gr(3,6)$,
\[
c_{(2,1),(2,1),(2,1)}^{(3,3,3)}=2,
\qquad
\mathcal S_{(2,1)}^3=2[\mathrm{pt}]
\]
so $ \mathcal S_{(2,1)}^3$ does not generate the top integral cohomology group.
\end{example}

For fixed $\sigma\in H_r$, \Cref{cor: Schubert corrected} says that
$\mathcal S_\sigma^r$ is nonzero in $H^*(\Gr(k,k+r),\mathbb Z)$ for every
$k\geq|\sigma|$.  This is a stabilization statement as the ambient
Grassmannian varies with $k$. It does not imply nonvanishing of analogous top-degree products in a fixed Grassmannian.

\begin{example}\label{ex: corrected vanishing comparison}
The partition $\sigma=(4,1,1)$ fits in the $3\times4$ rectangle and therefore
indexes a Schubert class in $\Gr(3,7)$.  Its Poincar\'e-dual complementary
partition is $(3,3)$, not $(4,1,1)$.  Hence
\[
c_{(4,1,1),(4,1,1)}^{(4,4,4)}=0,
\qquad
\mathcal S_{(4,1,1)}^2=0
\quad\text{in }H^*(\Gr(3,7),\mathbb Z).
\]
This is a genuine top-degree vanishing example.  In contrast, the same
partition indexes a class in $\Gr(6,10)$, and since $|(4,1,1)|=6$ and $r=4$,
\Cref{cor: Schubert corrected} gives
$\mathcal S_{(4,1,1)}^4\neq0$ in top degree.
\end{example}

\end{appendix}

	\bibliographystyle{alpha}
	\bibliography{References}

\end{document}